\documentclass[11pt,reqno]{amsart}
\usepackage{amsmath, amsthm, amssymb}
\usepackage{tensor}
\usepackage{bbm}
\usepackage{geometry}
\usepackage{soul}
\usepackage[pdftex]{graphicx}
\usepackage{float}
\usepackage{subfig}
\usepackage{tikz} 
\usetikzlibrary{arrows,decorations.pathreplacing,decorations.markings,shapes.geometric}
\tikzset{
    bt/.style={draw=blue,thick},
    ns/.style={circle,draw=blue,fill=blue, inner sep=0pt, minimum size=2mm},
    string/.style={draw=#1, postaction={decorate}, decoration={markings,mark=at position .45 with {\arrow[blue]{triangle 60}}}},
    doublestring/.style={draw=#1, postaction={decorate}, decoration={markings, mark=at position .7 with {\arrow[blue]{triangle 60}}, 
    mark=at position .3 with {\arrowreversed[blue]{triangle 60}}}},
    costring/.style={draw=#1, postaction={decorate}, decoration={markings,mark=at position .55 with {\arrow[draw=#1]{<}}}},
    arr/.style={string=blue, thick},
    doublearr/.style={doublestring=blue, thick},
    lin/.style={blue},
    dlin/.style = {blue, dashed, thick},
    dot/.style={circle,draw=#1,fill=#1,inner sep=1pt},
}

\usepackage{dsfont}
\usepackage[all]{xy}
\graphicspath{{graphics/}}
\numberwithin{equation}{section}
\usepackage{fancyhdr}
\theoremstyle{definition}

\newtheorem{Thm}{Theorem}[section]

\newtheorem{Lemma}{Lemma}[section]
\newtheorem{Rem}{Remark}[section]

\newcommand{\sh}[1]{\mathcal{#1}}

\DeclareMathOperator{\id}{id}

\DeclareMathOperator{\Uea}{\mathcal{U}}

\DeclareMathOperator{\SL}{SL}

\DeclareMathOperator{\C}{\mathbb{C}}

\DeclareMathOperator{\Z}{\mathbb{Z}}

\DeclareMathOperator{\diff}{d}

\DeclareMathOperator{\Norm}{N}

\DeclareMathOperator{\Lie}{Lie}

\DeclareMathOperator{\Proj}{\mathbb{P}}

\DeclareMathOperator{\Hom}{Hom}
\DeclareMathOperator{\length}{\ell}

\DeclareMathOperator{\Rder}{R}

\DeclareMathOperator{\Gm}{\mathbb{G}_m}

\DeclareMathOperator{\Ho}{H}

\DeclareMathOperator{\op}{op}

\DeclareMathOperator{\Dual}{\mathbb{D}}
\DeclareMathOperator{\Exten}{Ext}

\DeclareMathOperator{\Aff}{\mathbb{A}}

\DeclareMathOperator{\Lsimple}{L}
\DeclareMathOperator{\MVerma}{M}
\DeclareMathOperator{\RVerma}{R}

\DeclareMathOperator{\modulecat}{mod}

\DeclareMathOperator{\ac}{ac}

\usepackage[pdftex,
bookmarks=true,
bookmarksnumbered=true,
hypertexnames=false,
colorlinks = true,
linkcolor = black,
citecolor = black,
urlcolor = black]{hyperref}

\title{Cohomology of twisted $\sh{D}$-modules on $\Proj^1$
obtained as extensions from $\C^{\times}$}

\begin{document}

\author{Claude Eicher}
\address{DEPARTEMENT MATHEMATIK, ETH ZUERICH, 8092 ZUERICH, SWITZERLAND}
\email{claude.eicher@math.ethz.ch}
\date{\today}
\maketitle

\setcounter{tocdepth}{2}

\begin{abstract}
We construct twisted $\sh{D}$-modules on the projective line $\Proj^1$ that
are equivariant for the action of the diagonal torus subgroup of $\SL_2$. 
In the most interesting case these arise as extensions from local systems on $\C^{\times}$. 
We discuss their subquotient
structure. Their sheaf cohomology groups are weight modules for
the Lie algebra $\mathfrak{sl}_2$. We also discuss
their subquotient structure and in case these
modules are not the familiar highest or lowest weight modules, we give
an explicit presentation for them. Our computations illustrate some basic
$\sh{D}$-module concepts and the Beilinson-Bernstein equivalence.  
They are the first step in a program that aims to describe categories of
modules over semisimple and affine Kac-Moody Lie algebras that are next to highest (or lowest)
weight via $\sh{D}$-modules on the flag variety. 
\end{abstract}

\tableofcontents

\section{Introduction and Organization}
Consider a local system on $\C^{\times}$ given
by the monodromy $e^{2\pi i\alpha}$, where $\alpha \in \C$ is a parameter. It can be described by the
differential equation $(x\partial_x - \alpha)f(x)=0$
and hence as a connection on $\C^{\times}$.
The formulation of these objects in the language
of $\sh{D}$-modules is given in \autoref{ssec:defofOmegaalpha}.
The theory of $\sh{D}$-modules
provides extensions (direct images) of such objects, living on $\C^{\times}$, to $\sh{D}$-modules on 
the projective line $\Proj^1$. The simplest such extension
is the mere sheaf-direct image, while the $!$-extension
is constructed from it via a duality functor (\autoref{ssec:holonomicduality}).
In addition, with the aim of being systematic, we introduce
yet another extension denoted $j_{!x\cdot z}$ in \autoref{ssec:directimagesofOmegaalpha}. 
We analyze all these extensions in \autoref{ssec:directimagesofOmegaalpha} by determining their subquotients
(composition series) in the category of $\sh{D}$-modules. 
In addition to $\C^{\times}$ we have also included for comparison in \autoref{ssec:directimagesofOmegaCxz} the easier and even better known case
of the affine line. 
As a technical aside, let us mention that we work with right $\sh{D}$-modules
starting from \autoref{ssec:directimagesofOmegaCxz} because the formation of direct images is easier than
for left $\sh{D}$-modules.\par

In \autoref{sec:H0andH1} we compute the sheaf cohomology groups
of our extensions as $\mathfrak{sl}_2$-modules using the
Cech complex for the standard affine open cover of $\Proj^1$. They turn out to be weight modules (\autoref{sec:generalitiesonsl2modules})
for $\mathfrak{sl}_2$ with trivial central character and lowest or highest weight modules in case of the affine line. 
The identification of these cohomology groups can be understood as a building block in the
description of the category of $\mathfrak{sl}_2$-weight modules with trivial central character via 
$\sh{D}$-modules on $\Proj^1$. The basic reason 
for the appearence of the subspace $\C^{\times}$ is that $\SL_2$ acts on $\Proj^1$ and
$\C^{\times}\sqcup \{0\}\sqcup \{\infty\}$ is the stratification of $\Proj^1$ by orbits of the diagonal torus subgroup
of $\SL_2$. Similarly, the affine lines $\Proj^1\setminus\{\infty\}$ resp. 
$\Proj^1\setminus\{0\}$ are the open orbit of the Borel resp. opposite Borel subgroup of $\SL_2$.  
The corresponding $\sh{D}$-module extensions (the $!$- resp. $*$-extension is usually called standard resp. costandard object
in this context) acquire an equivariance property w.r.t. the subgroup. The equivariance is addressed in \autoref{ssec:Gmequivariance} and in Remark \autoref{Rem:BBuntwisted}. 
Further, we discuss in \autoref{ssec:autoequivalences-} how the change of variables $x \mapsto x^{-1}$ induces auto-equivalences
$(\cdot)^-$ of the category of $\sh{D}$- resp. $\mathfrak{sl}_2$-modules.
\par

In \autoref{sec:twistedDmodules} we review the notion of twisted $\sh{D}$-modules. In our
case the twist parameter is an arbitrary integer $\lambda$. Subsequently we determine the sheaf cohomology groups of our extensions considered now
as twisted $\sh{D}$-modules, thereby generalizing the results of \autoref{sec:H0andH1}. They are again weight modules for $\mathfrak{sl}_2$ with
central character depending on $\lambda$.
In \autoref{ssec:descriptionofRVermalambdaalpha} we analyze the structure (composition series) of a weight module (depending on the parameters $\alpha$ and $\lambda$),
which appears as global sections of our extensions. 
The main results of this work are gathered into Theorems \autoref{Thm:hwmodulestwisted}, \autoref{Thm:jdotOmegaalpha}, \autoref{Thm:jdotOmega},
\autoref{Thm:j!Omegaalpha} and \autoref{Thm:j!dottwisted}. These are all proven by direct computations with the Cech complex. 
While parts of these results have previously appeared in the literature, our treatment considers all choices of the parameter $\alpha$
and all integral twists $\lambda$ for $\sh{D}$-modules. As explained in Remarks \autoref{Rem:BBuntwisted} and \autoref{Rem:BBtwisted}, the
results illustrate the Beilinson-Bernstein equivalence.\par

Finally, we sketch in \autoref{sec:outlook} how this work fits into a larger program by replacing $\SL_2$ by any
semisimple algebraic group over $\C$ or the affine Kac-Moody Lie algebra $\widehat{\mathfrak{sl}_2}$. 

\section*{Notation}
We use the symbol \S\ when referring to
entire sections. When making a statement that holds for multiple indices separately, for brevity we simply list
the indices, see e.g. \autoref{ssec:geometricsetup}: we write $\C_{x,z}$
instead of $\C_x, \C_z$.
We denote by curly letters sheaves
$\sh{F}$ and their sections over an open $U$ by $\sh{F}(U)$. 
If $G$ is a linear algebraic group over $\C$, then $\Lie G$ denotes its Lie algebra. 
All varieties we consider are algebraic varieties over $\C$. By a local system, we mean a vector bundle with flat connection.
We have tried to isolate statements that do not depend strongly on the rest of the text
or are used multiple times into remarks. 

\section*{Acknowledgements}
I am indebted to my supervisor G. Felder for his guidance during the preparation of this work,
for letting me present the main results to him and for his comments on the manuscript. 

\section{Untwisted $\sh{D}$-modules}\label{sec:untwistedDmodules}
\subsection{Left and right $\sh{D}$-modules}\label{ssec:leftandrightDmod}
We recall a few generalities on left and right $\sh{D}$-modules, see e.g. \cite{Kas00}[section 1.4]. 
Let $X$ be a smooth variety. Let $\sh{D}_X$ be the sheaf of 
algebraic differential operators on $X$. We denote by
$\sh{D}_X\mod$ resp. $\mod \sh{D}_X$ the category of 
left resp. right $\sh{D}_X$-modules quasicoherent as $\sh{O}_X$-modules. 
Then the structure sheaf $\sh{O}_X$ and
the sheaf of volume forms $\Omega_X$ is naturally a left resp. right $\sh{D}_X$-module. 
Further we have an isomorphism of sheaves of $\C$-algebras 
\begin{align}\label{eq:DXop}
\sh{D}_X^{\op} \xrightarrow{\cong} \Omega_X \otimes_{\sh{O}_X}\sh{D}_X\otimes_{\sh{O}_X} \Omega_X^{\otimes -1}\;.
\end{align} 
In (\'etale) local coordinates $(x_j)_{1 \leq j \leq n}$ on $X$ it is given by $p \mapsto \diff x \otimes  ^tp \otimes \diff x^{\otimes -1}$.
Here $\diff x=\diff x_1 \wedge \dots \wedge\diff x_n$ and $\diff x^{\otimes -1}$ satisfies $\langle \diff x, \diff x^{\otimes -1}\rangle = 1$.
The \emph{formal adjoint}  $^t(\cdot)$ is the unique anti-involution of $\sh{D}_X$ satisfying 
$^t(x_j)=x_j$ and $^t(\partial_{x_j})=-\partial_{x_j}$. Via the isomorphism $\eqref{eq:DXop}$ it becomes
manifest that $\Omega_X$ is a left $\sh{D}_X^{\op}$-module and that a left $\sh{D}_X^{\op}$-module
is the same as a $\Omega_X$-twisted left $\sh{D}_X$-module
(the notion of a twisted $\sh{D}$-module will be recalled 
in \autoref{ssec:generalitiesontwistedDmodules} below). The functor 
$(\cdot)^r: \sh{D}_X\modulecat \rightarrow \modulecat\sh{D}_X$ given on objects by 
$\sh{M} \mapsto \sh{M}^r = \Omega_X \otimes_{\sh{O}_X} \sh{M}$
is an exact equivalence. 
Similarly, the functor $(\cdot)^l: \modulecat\sh{D}_X \rightarrow \sh{D}_X\modulecat$, 
$\sh{M} \mapsto \sh{M}^l = \Omega_X^{\otimes -1} \otimes_{\sh{O}_X} \sh{M}$,
is an exact equivalence inverse to $(\cdot)^r$.

\subsection{Holonomic duality $\Dual$}\label{ssec:holonomicduality}
For the notions of this subsection, see e.g. \cite{Kas00}[section 4.11]. 
Let us assume $\dim X = 1$. 
Let $\sh{M}$ be a holonomic right $\sh{D}_X$-module (in particular $\sh{M}$ is required to be coherent). 
Its dual is defined as $\Dual \sh{M} = \sh{E}xt^1_{\sh{D}_X}(\sh{M},\sh{D}_X)^r$.
Here $\sh{E}xt^1_{\sh{D}_X}(\sh{M},\sh{D}_X)$ is naturally a left $\sh{D}_X$-module, to which we apply $(\cdot)^r$.
$\Dual$ defines an exact contravariant auto-equivalence of the category of holonomic $\sh{D}_X$-modules. 

\begin{Rem}\label{Rem:formulaforDualM} We recall a formula for $\Dual \sh{M}$
in case of $X=\C$, the affine line, and $\sh{M}=\sh{D}_{\C}/p\sh{D}_{\C}$,
where $p$ is a nonzero polynomial in $x$ and $\partial_x$.
Consider the free resolution of right $\sh{D}_{\C}$-modules
\begin{align}\label{eq:freeres}
0 \rightarrow \sh{D}_{\C} \xrightarrow{p\cdot} \sh{D}_{\C} \rightarrow \sh{M} \rightarrow 0\;.
\end{align}
The map $\sh{H}om_{\sh{D}_X}(\sh{D}_X,\sh{D}_X) \xrightarrow{\cong} \sh{D}_X$,
$\phi \mapsto \phi(1)$, is an isomorphism of left $\sh{D}_X$-modules. We apply $\sh{H}om_{\sh{D}_{\C}}(\cdot,
\sh{D}_{\C})$ to \eqref{eq:freeres} and get an exact sequence of left $\sh{D}_{\C}$-modules
\begin{align*}
0 \leftarrow \sh{D}_{\C} \xleftarrow{\cdot p} \sh{D}_{\C} \leftarrow \sh{H}om_{\sh{D}_{\C}}(\sh{M},\sh{D}_{\C}) \leftarrow 0\;.
\end{align*}
It follows  $\sh{H}om_{\sh{D}_{\C}}(\sh{M},\sh{D}_{\C})=0$ and $\sh{E}xt^1_{\sh{D}_{\C}}(\sh{M},\sh{D}_{\C}) \cong \sh{D}_{\C}/\sh{D}_{\C}p$.
With \eqref{eq:DXop} we conclude that $\Omega_{\C}\otimes_{\sh{O}_{\C}}\sh{D}_{\C} \rightarrow
\sh{D}_{\C} \otimes_{\sh{O}_{\C}}\Omega_{\C}$, $\diff x\otimes p \mapsto ^t p\otimes \diff x$,
is an isomorphism of left $\sh{D}_{\C}^{\op}$-modules and hence
\begin{align*}
\Dual\sh{M} =\Omega_{\C} \otimes_{\sh{O}_{\C}}\left(\sh{D}_{\C}/\sh{D}_{\C}p\right)
=  \sh{D}_{\C}\otimes_{\sh{O}_{\C}}\Omega_{\C} /  (^tp)\sh{D}_{\C}\otimes_{\sh{O}_{\C}}\Omega_{\C} \cong \sh{D}_{\C}/ ^tp\sh{D}_{\C}
\end{align*}
is an isomorphism of right $\sh{D}_{\C}$-modules. This formula for $\Dual \sh{M}$ will be applied later.
\end{Rem}

\subsection{Geometric setup}\label{ssec:geometricsetup}
In the remaining part of the text we will be concerned
with the following geometric setup. Let $x$ be a global coordinate on $\Proj^1$
considered as a variety.
Then $z = x^{-1}$ is another global coordinate. Consider the open subsets $\C_x = (x \neq \infty)$
and $\C_z = (z \neq \infty)$
of $\Proj^1$ and the corresponding open embeddings $j_{x,z}: \C_{x,z} \hookrightarrow \Proj^1$
and $j: \C^{\times}=\C_x\cap \C_z \hookrightarrow \Proj^1$. Thus, we have a commutative diagram of
open embeddings
\begin{align*}
\xymatrix{& \C_x \ar@{^{(}->}[dr] & \\
\C^{\times} \ar@{^{(}->}[ur] \ar@{^{(}->}[dr] &  & \Proj^1\\
& \C_y \ar@{^{(}->}[ur] & }\;.
\end{align*}
The closed embeddings of the complements will be denoted by $\iota_{x,z}$.
We will also work with homogeneous coordinates $[z_0:z_1]$ on $\Proj^1$
related to the coordinate $x$ by $x=z_1/z_0$.

\subsection{Direct images of $\Omega_{\C_{x,z}}$}\label{ssec:directimagesofOmegaCxz}
\emph{In the subsequent text we will only consider right $\sh{D}$-modules unless
stated otherwise.} First recall that the action of $\partial_x$
on $\Omega_{\C_x}(\C_x)=\C[x]\diff x$ is given by $(x^n \diff x)\partial_x = -nx^{n-1}\diff x$. 
Also recall that the sheaf direct image $j_{x\cdot}\Omega_{\C_x}$
of $\Omega_{\C_x}$ is naturally a $\sh{D}_{\Proj^1}$-module.
Let us describe it on the open affine cover $\{\C_x,\C_z\}$ of $\Proj^1$. 
We have $(j_{x\cdot}\Omega_{\C_x})\vert \C_z \cong \sh{D}_{\C_z}/z\partial_z \sh{D}_{\C_z}$
as 
\begin{align*}
(j_{x\cdot}\Omega_{\C_x})(\C_z)=\Omega_{\C_x}(\C_z\cap \C_x)=\C[x,x^{-1}]\diff x
=\C[z,z^{-1}]\diff z \xrightarrow{\cong} (\sh{D}_{\C_z}/z\partial_z \sh{D}_{\C_z})(\C_z)\;.
\end{align*}
The last isomorphism is given by $\frac{\diff z}{z} \mapsto \overline{1}$. Here
and in the following $\overline{a}$ will denote the class of an element $a$ in a quotient.
We have the exact sequence
\begin{align}\label{eq:sesonCz}
0 \rightarrow \Omega_{\C_z}=\sh{D}_{\C_z}/\partial_z\sh{D}_{\C_z}
\rightarrow (j_{x \cdot}\Omega_{\C_x})\vert \C_z = \sh{D}_{\C_z}/z\partial_z\sh{D}_{\C_z}
\rightarrow \sh{D}_{\C_z}/z\sh{D}_{\C_z} \rightarrow 0
\end{align}
of $\sh{D}_{\C_z}$-modules. 
The first map sends $\overline{1} \mapsto \overline{z}$.  
Trivially, $(j_{x\cdot}\Omega_{\C_x})\vert \C_x = \Omega_{\C_x}$.
Similarly, $(j_{z\cdot}\Omega_{\C_z})\vert \C_x \cong \sh{D}_{\C_x}/x\partial_x \sh{D}_{\C_x}$.

\begin{Rem} Let us recall the exact triangle from a pair of complementary closed and open embedding. 
Let $\iota: Z \hookrightarrow X$
be a closed immersion of a smooth variety $Z$ into a smooth variety $X$. Let $j: U=X\setminus Z \hookrightarrow X$
be the open embedding of the complement.
If $\sh{M}$ is an object of the bounded derived category of $\sh{D}_X$-modules
we have an exact triangle \cite{Ber}[p. 8]
\begin{align}\label{eq:exacttriangle}
\iota_* \iota^! \sh{M} \rightarrow \sh{M} \rightarrow j_{\cdot}j^{-1}\sh{M} \xrightarrow{[1]}\;.
\end{align}
We comment on the direct and inverse images involved: 
Recall that $j^{-1}$, the restriction to the open $U$, is exact.
$\Rder j_{\cdot} \equiv j_{\cdot}$ is the right derived direct image of sheaves.
In \eqref{eq:exacttriangle} $\sh{M} \rightarrow j_{\cdot}j^{-1}\sh{M}$ is the unit
of the adjunction. According to \cite{Ber}[p. 5] $\iota_*$ is exact and $\Rder \iota^! \equiv \iota^!$ is right derived
since $\iota$ is a closed embedding. $\iota_*$ is left adjoint to $\iota^!$ and 
in \eqref{eq:exacttriangle} $\iota_* \iota^! \sh{M} \rightarrow \sh{M}$
is the counit of the adjunction.
$\iota_* \iota^! = \Rder \Gamma_Z$ is the right derived functor of sections
supported in $Z$. Also recall from \cite{Kas00}[section 3.4] that if $\sh{M}=\Omega_X$ (placed in degree zero) and $Z$ is
a complete intersection of codimension $d$ in $X$ then
$\Ho^l(\Rder \Gamma_Z(\Omega_X))=0$ for $l \neq d$ and $\sh{B}_{Z \vert X}=\Ho^d(\Rder \Gamma_Z(\Omega_X))$
is the $\sh{D}$-module of distributions on $X$ with singularities along $Z$.\end{Rem}

Now we consider the above remark for $\iota=\iota_x$, $j=j_x$ and $\sh{M}=\Omega_{\Proj^1}$ (in degree zero).
We find $\iota_x^! \Omega_{\Proj^1} = \C[-1]$. (Clearly $\Omega_{\Proj^1}$ does not have
sections supported at $(x=\infty)$.) The long exact sequence associated to \eqref{eq:exacttriangle} is
\begin{align}\label{eq:sesforjx}
0 \rightarrow \Omega_{\Proj^1} \rightarrow j_{x\cdot}\Omega_{\C_x} \rightarrow \iota_{x*}\C \rightarrow 0 \rightarrow \dots\;.
\end{align}
The restriction of this sequence to the open $\C_z$ is \eqref{eq:sesonCz}. Here $\iota_{x*}\C = \Ho^1(\Rder \Gamma_{(x=\infty)}(\Omega_{\Proj^1}))=\sh{B}_{(x=\infty)\vert \Proj^1}$ is the skyscraper $\sh{D}$-module supported at $x=\infty$.

\subsubsection{Subquotients of $j_{x\cdot}\Omega_{\C_x}$} These are clear from \eqref{eq:sesforjx} and the fact that $\Omega_{\Proj^1}$ and $\iota_{x*}\C$ are simple $\sh{D}_{\Proj^1}$-modules.

\subsubsection{Subquotients of $j_{x!}\Omega_{\C_x}$} 
The direct image $j_{x!}$ is defined as $j_{x!}=\Dual j_{x\cdot}\Dual$,
see e.g. \cite{Ber}[p. 18]. Applying $\Dual$ to \eqref{eq:sesforjx} we find with 
$\Dual \Omega_{\Proj^1} \cong \Omega_{\Proj^1}$ and $\Dual \iota_{x*}\C \cong \iota_{x*}\C$ the exact sequence
\begin{align}\label{eq:jx!exactsequence}
0 \leftarrow \Omega_{\Proj^1} \leftarrow j_{x!}\Omega_{\C_x} \leftarrow  \iota_{x*}\C \leftarrow 0\;.
\end{align}
According to \eqref{eq:sesonCz} and Remark \autoref{Rem:formulaforDualM} this sequence restricts to
\begin{align}\label{eq:sesforjxshriek}
0 \leftarrow \Omega_{\C_z}=\sh{D}_{\C_z}/\partial_z \sh{D}_{\C_z} \leftarrow (j_{x!}\Omega_{\C_x})\vert \C_z =\sh{D}_{\C_z}/\partial_z z\sh{D}_{\C_z} \leftarrow \sh{D}_{\C_z}/z\sh{D}_{\C_z} \leftarrow 0
\end{align}
on $\C_z$. The first map (from the right) sends $\overline{1} \mapsto \overline{\partial_z}$.

\subsection{Definition of $\Omega^{(\alpha)}_{\C^{\times}}$}\label{ssec:defofOmegaalpha}
For $\alpha \in \C$ define $\Omega_{\C^{\times}}^{(\alpha)} = \sh{D}_{\C^{\times}}/(x\partial_x-\alpha)\sh{D}_{\C^{\times}}$.
This definition depends on the choice of the global coordinate $x$.
In the global coordinate $z=x^{-1}$ we find $\Omega_{\C^{\times}}^{(\alpha)}=\sh{D}_{\C^{\times}}/(z\partial_z+\alpha)\sh{D}_{\C^{\times}}$.
As a special case we have 
\begin{align*}
\Omega_{\C^{\times}}^{(-1)} = \sh{D}_{\C^{\times}}/\partial_x x\sh{D}_{\C^{\times}}
= \sh{D}_{\C^{\times}}/\partial_x \sh{D}_{\C^{\times}}= \Omega_{\C^{\times}}\;.
\end{align*}

\begin{Rem}\label{Rem:eltarypropertiesofOmegaalpha} We list elementary properties of $\Omega^{(\alpha)}_{\C^{\times}}$.
\begin{enumerate}
\item \emph{$\Omega^{(\alpha)}_{\C^{\times}}$ is a free $\sh{O}_{\C^{\times}}$-module of rank one.} Indeed,
$\sh{D}_{\C^{\times}}(\C^{\times})/(x\partial_x-\alpha)\sh{D}_{\C^{\times}}(\C^{\times})
= \bigoplus_{n \in \Z}\C\overline{1}x^n$. Since $(\overline{1}x^n) x\partial_x = (\alpha-n)\overline{1}x^n$ our module is a direct sum of one dimensional eigenspaces of $x\partial_x$ with eigenvalues $\alpha + \Z$.

\item \emph{$\Omega_{\C^{\times}}^{(\alpha)} \cong \Omega_{\C^{\times}}^{(\beta)}$
if and only if $\alpha-\beta \in \Z$.} From the eigenspace decomposition
for $x\partial_x$ described in (1) we find that $\alpha-\beta \notin \Z$
implies $\Omega_{\C^{\times}}^{(\alpha)} \ncong \Omega_{\C^{\times}}^{(\beta)}$. On the other hand, for $n \in \Z$
$\phi: \Omega_{\C^{\times}}^{(\alpha)} \rightarrow \Omega_{\C^{\times}}^{(\alpha+n)}$, $\overline{1} \mapsto \overline{1}x^n$, extends uniquely to a morphism of $\sh{D}_{\C^{\times}}$-modules because of $\phi(\overline{1}x\partial_x)=\phi(\overline{1})x\partial_x = \alpha \phi(\overline{1})$. Further, $\phi$ is invertible.

\item $\Dual \Omega^{(\alpha)}_{\C^{\times}} \cong \sh{H}om_{\sh{O}_{\C^{\times}}}
(\Omega^{(\alpha)}_{\C^{\times}},\Omega_{\C^{\times}})^r \cong \Omega^{(-\alpha)}_{\C^{\times}}$.
The first isomorphism is known to hold, as consequence of (1), see e.g. \cite{Kas00}[Lemma 3.13]. 
We now explain the second isomorphism. 
Recall that $\sh{H}om_{\sh{O}_{\C^{\times}}}
(\Omega^{(\alpha)}_{\C^{\times}},\Omega_{\C^{\times}})$ is a left $\sh{D}_{\C^{\times}}$-module
if we let the vector field $\xi \in \sh{D}_{\C^{\times}}$ act on a local section $\psi$ by
\begin{align}\label{eq:dualconnection}
(\xi \psi)(m)=-\psi(m)\xi+\psi(m\xi)\;,
\end{align}
see e.g. \cite{Bor87}[Chapter VI, 3.4].
This implies that the map between the right $\sh{D}_{\C^{\times}}$-modules
\begin{align*}
\Omega_{\C^{\times}}\otimes_{\sh{O}_{\C^{\times}}}\sh{H}om_{\sh{O}_{\C^{\times}}}(\Omega^{(\alpha)}_{\C^{\times}},\Omega_{\C^{\times}})
\rightarrow \Omega_{\C^{\times}},\; \diff x\otimes \psi \mapsto \psi(\overline{1})\;,
\end{align*}
which is an isomorphism of $\sh{O}_{\C^{\times}}$-modules, sends 
\begin{align*}
(\diff x\otimes \psi)x\partial_x\; \mapsto\; \psi(\overline{1})(-1-\alpha+x\partial_x)\;,
\end{align*}
where we applied \eqref{eq:dualconnection}. 
Hence, the global sections of $\sh{H}om_{\sh{O}_{\C^{\times}}}(\Omega^{(\alpha)}_{\C^{\times}},\Omega_{\C^{\times}})^r$
decompose into a direct sum of one dimensional eigenspaces for the eigenvalues $-\alpha+\Z$ for the action of $x\partial_x$.
Hence, by (2), the second isomorphism.  Also note that $\Dual \Omega^{(\alpha)}_{\C^{\times}}
\cong \Omega^{(-\alpha)}_{\C^{\times}}$ follows directly from remark \autoref{Rem:formulaforDualM} and (2)
as $^t(x\partial_x-\alpha)=-x\partial_x-\alpha-1$.
\end{enumerate}
\end{Rem}

\subsection{Direct images of $\Omega^{(\alpha)}_{\C^{\times}}$}\label{ssec:directimagesofOmegaalpha}
Let us describe the direct images $j_{\cdot}$
and $j_!$ of $\Omega^{(\alpha)}_{\C^{\times}}$.
In addition, we will consider a functor $j_{!x\cdot z}: \modulecat\sh{D}_{\C^{\times}}
\rightarrow \modulecat\sh{D}_{\Proj^1}$ defined on objects as follows.
Let $\sh{M}$ be a $\sh{D}_{\C^{\times}}$-module. We glue the $\sh{D}_{\C_x}$-module
$j_x^{-1}j_!\sh{M}$ and the $\sh{D}_{\C_z}$-module $j_z^{-1}j_{\cdot}\sh{M}$ along $\C^{\times}$ to 
a $\sh{D}_{\Proj^1}$-module $j_{!x \cdot z}\sh{M}$. Thus, informally $j_{!x\cdot z}$ can be described
as the $!$-extension at $x=0$ and the $\cdot$-extension at $z=0$. 
Of course we have a similar functor $j_{\cdot x!z}$ satisfying $j_{\cdot x !z}=\Dual j_{!x\cdot z} \Dual$. 

\begin{Lemma}\label{Lemma:alphanotanintegerjdotsimple}
Let $\alpha \notin \Z$.
\begin{enumerate}
\item $j_{\cdot}\Omega^{(\alpha)}_{\C^{\times}}$ is simple. The restriction to $\C_x$ is
\begin{align}\label{eq:jdotOmegaalphaonCx}
j_{\cdot}\Omega^{(\alpha)}_{\C^{\times}}\vert \C_x \cong \sh{D}_{\C_x}/(x\partial_x-\alpha)\sh{D}_{\C_x}\;.
\end{align} The canonical
map $j_! \Omega^{(\alpha)}_{\C^{\times}} \rightarrow j_{\cdot}\Omega^{(\alpha)}_{\C^{\times}}$, see e.g. \cite{Ber}[p. 18],
is an isomorphism. 
\item $j_{!x\cdot z}\Omega^{(\alpha)}_{\C^{\times}} \cong j_{\cdot}\Omega^{(\alpha)}_{\C^{\times}}$
\end{enumerate}
\end{Lemma}
\begin{proof}
(1). We have the decomposition 
\begin{align*}
j_{\cdot}\Omega^{(\alpha)}_{\C^{\times}}(\C_x)=\bigoplus_{n \in \Z}\left(j_{\cdot}\Omega^{(\alpha)}_{\C^{\times}}(\C_x)\right)_{\alpha+n}
=\bigoplus_{n \in \Z}\left(\Omega^{(\alpha)}_{\C^{\times}}(\C^{\times})\right)_{\alpha+n}
\end{align*} 
into one dimensional eigenspaces for $x\partial_x$ and $x$ and $\partial_x$ intertwine those:
\begin{align*}
\xymatrix{ \left(\Omega^{(\alpha)}_{\C^{\times}}(\C^{\times})\right)_{\alpha+n+1} \ar@<5pt>[r]^{\cdot x}
& \ar[l]^{\cdot \partial_x} \left(\Omega^{(\alpha)}_{\C^{\times}}(\C^{\times})\right)_{\alpha+n}}\;.
\end{align*}
As  
\begin{align*}
\cdot x\partial_x =(\alpha+n+1)\id_{\left(\Omega^{(\alpha)}_{\C^{\times}}(\C^{\times})\right)_{\alpha+n+1}}
\qquad \cdot\partial_x x = (\alpha+n+1)\id_{\left(\Omega^{(\alpha)}_{\C^{\times}}(\C^{\times})\right)_{\alpha+n}}
\end{align*}
the maps $x$ and $\partial_x$ are invertible for each $n$ if $\alpha \notin \Z$.
Let $\sh{M}$ be a submodule of $j_{\cdot}\Omega^{(\alpha)}_{\C^{\times}}$.
Let $s \in \sh{M}(\C_x)$. Writing $s = \sum_{j=1}^n s_j$ with
$s_j x\partial_x = \lambda_j s_j$ and distinct $\lambda_j$ we conclude $s_j \in \sh{M}(\C_x)$ by considering $s(x\partial_x)^l$,
$0 \leq l \leq n-1$. Thus $\sh{M}(\C_x)=\bigoplus_{n \in \Z}\sh{M}(\C_x)_{\alpha+n}$.
$x$ and $\partial_x$ restrict to isomorphisms between $\sh{M}(\C_x)_{\alpha+n+1}$
and $\sh{M}(\C_x)_{\alpha+n}$. It follows $\sh{M}(\C_x)=0$
or $\sh{M}(\C_x)=j_{\cdot}\Omega^{(\alpha)}_{\C^{\times}}(\C_x)$. Repeating the
argument for $\C_z$ we find $\sh{M} = 0$ or $\sh{M}=j_{\cdot}\Omega^{(\alpha)}_{\C^{\times}}$. 
We have shown that $j_{\cdot}\Omega^{(\alpha)}_{\C^{\times}}$ is simple. Further it is clear 
that $j_{\cdot}\Omega^{(\alpha)}_{\C^{\times}}\vert \C_x \cong \sh{D}_{\C_x}/(x\partial_x-\alpha)\sh{D}_{\C_x}$.
It follows with $j_!=\Dual j_{\cdot}\Dual$, remark \autoref{Rem:formulaforDualM} and remark \autoref{Rem:eltarypropertiesofOmegaalpha}
(2) that $j_!\Omega^{(\alpha)}_{\C^{\times}}\vert \C_x \cong \sh{D}_{\C_x}/(x\partial_x-\alpha)\sh{D}_{\C_x}$.
Thus $j_!\Omega^{(\alpha)}\cong j_{\cdot}\Omega^{(\alpha)}$.
Recall that the canonical map $\phi: j_! \Omega^{(\alpha)}_{\C^{\times}} \rightarrow j_{\cdot}\Omega^{(\alpha)}_{\C^{\times}}$ restricts
to the identity on $\C^{\times}$. Thus $\ker\phi$ is supported on $(x=0)\sqcup (z=0)$. 
Since we know that $j_!\Omega^{(\alpha)}$ is simple it follows $\ker\phi=0$. Because $j_!\Omega^{(\alpha)}\cong j_{\cdot}\Omega^{(\alpha)}$
it follows that $\phi$ also surjects.\newline
(2). By definition we have $j_{!x\cdot z}\Omega^{(\alpha)}_{\C^{\times}}\vert \C_z \cong j_{\cdot}\Omega^{(\alpha)}_{\C^{\times}}\vert \C_z$
and $j_{!x\cdot z}\Omega^{(\alpha)}_{\C^{\times}} \vert \C_x \cong j_! \Omega^{(\alpha)}_{\C^{\times}}\vert \C_x$ restricting to the identity of $\Omega^{(\alpha)}_{\C^{\times}}$ on $\C^{\times}$. Further, by (1)
we have an isomorphism $j_!\Omega^{(\alpha)}_{\C^{\times}}\vert \C_x \xrightarrow{\cong} j_{\cdot}\Omega^{(\alpha)}_{\C^{\times}}\vert \C_x$
restricting to the identity of $\Omega^{(\alpha)}_{\C^{\times}}$ on $\C^{\times}$. The assertion follows. 
\end{proof}

As a consequence of this lemma we are left with determining the subquotients of the
direct images of $\Omega_{\C^{\times}}$.

\subsubsection{Subquotients of $j_{\cdot}\Omega_{\C^{\times}}$
and $j_!\Omega_{\C^{\times}}$}\label{sssec:subquotientsjdotandj!}
$j_{\cdot}\Omega_{\C^{\times}}$ fits into the exact sequence
\begin{align}\label{eq:sesforjdotOmegaCtimes}
0 \rightarrow j_{x\cdot}\Omega_{\C_x} \rightarrow j_{\cdot}\Omega_{\C^{\times}} \rightarrow \iota_{z*}\C \rightarrow 0
\end{align}
and into the one obtained from it by interchanging $x$ and $z$. Combining this with \eqref{eq:sesforjx} we obtain a diagram of embeddings $A \hookrightarrow B$ such that $B/A$ is simple
\begin{align*}
\xymatrix{ & j_{\cdot}\Omega_{\C^{\times}} & \\ j_{x\cdot}\Omega_{\C_x} \ar@{^{(}->}[ur] & & j_{z\cdot}\Omega_{\C_z} \ar@{_{(}->}[ul]\\
& \ar@{^{(}->}[ul] \Omega_{\Proj^1} \ar@{_{(}->}[ur]&}\;.
\end{align*}
Thus the simple objects $\Omega_{\Proj^1}$, $\iota_{x*}\C$ and $\iota_{z*}\C$ each have multiplicity one in $j_{\cdot}\Omega_{\C^{\times}}$. This describes the subquotient structure of $j_{\cdot}\Omega_{\C^{\times}}$. 
Applying $\Dual$ to \eqref{eq:sesforjdotOmegaCtimes} yields
\begin{align}\label{eq:sesforjshriekOmegaCtimes}
0 \leftarrow j_{x!}\Omega_{\C_x} \leftarrow j_! \Omega_{\C^{\times}} \leftarrow \iota_{z*}\C \leftarrow 0\;.
\end{align}
Combining this with \eqref{eq:jx!exactsequence} we see that the diagram of embeddings
\begin{align*}
\xymatrix{ & j_! \Omega_{\C^{\times}} & \\
& \ar@{_{(}->}[u] \iota_{x*}\C\oplus \iota_{z*}\C & \\
\iota_{x*}\C \ar@{^{(}->}[ur] & & \iota_{z*}\C \ar@{_{(}->}[ul] }
\end{align*}
has simple quotients. 

\subsubsection{Subquotients of $j_{!x\cdot z}\Omega_{\C^{\times}}$}
From the definition of $j_{!x\cdot z}\Omega_{\C^{\times}}$ and \autoref{sssec:subquotientsjdotandj!} we find the exact sequences
\begin{align}\label{eq:sesforjdotx!zOmegaCtimes}
\begin{split}
& 0 \rightarrow j_{z!}\Omega_{\C_z} \rightarrow j_{!x\cdot z}\Omega_{\C^{\times}} \rightarrow \iota_{x*}\C \rightarrow 0\\
& 0 \rightarrow \iota_{z*}\C \rightarrow j_{!x\cdot z}\Omega_{\C^{\times}} \rightarrow j_{x\cdot}\Omega_{\C_x} \rightarrow 0\;.
\end{split}
\end{align}
We have a diagram of embeddings
\begin{align*}
\iota_{z*}\C \hookrightarrow j_{z!}\Omega_{\C_z} \hookrightarrow j_{!x\cdot z}\Omega_{\C^{\times}}
\end{align*}
with simple quotients. Again, $\Omega_{\Proj^1}$, $\iota_{x*}\C$ and $\iota_{z*}\C$ each have multiplicity one in $j_{!x\cdot z}\Omega_{\C^{\times}}$.

\subsection{$\Gm$-Equivariance}\label{ssec:Gmequivariance} The multiplicative group $\Gm$ acts on $\Proj^1$ 
by $t [z_0 : z_1] = [t z_0 : t^{-1} z_1]$. 
The orbits are $(x=0)$, $(z=0)$ and $\C^{\times}$. We describe the
equivariance of the constructed $\sh{D}$-modules w.r.t. this action
in the sense of \cite{BB93}[section 1.8]. We follow the terminology of \cite{Gai05}
and remark that in \cite{Kas08} weak equivariance is called quasi-equivariance. To this end we note
that if $Y=\C^{\times}, \C_x$ or $\Proj^1$ then $\sh{D}_Y$
is naturally a weakly $\Gm$-equivariant $\sh{D}_Y$-module.
If $\sh{I} \subseteq \sh{D}_Y$ is a sheaf of right ideals,
then $\sh{D}_Y / \sh{I}$ has an induced weakly equivariant structure
if and only if $t\sh{I} \subseteq \sh{I}$ for all $t \in \Gm$. 
Also $\Omega_Y$ has a natural strongly $\Gm$-equivariant structure. Let $\sh{M}$ 
be any holonomic $\sh{D}_Y$-module. 
\subsubsection{Equivariance of the holonomic dual}\label{ssec:equivarianceofholonomicdual}
It is shown, in greater generality, in \cite{RS13}[Proposition 2.18]
that there is an induced weakly $\Gm$-equivariant structure
on $\Dual\sh{M}$ for any weakly $\Gm$-equivariant holonomic
$\sh{D}_Y$-module $\sh{M}$. 
Further, it is stated in \cite{BB93}[section 2.5.8] that $\Dual$ is
a duality on the category of strongly $\Gm$-equivariant holonomic $\sh{D}$-modules
(put $H=1$ in loc. cit.).  

\subsubsection{Equivariance of $\Omega^{(\alpha)}_{\C^{\times}}$}
Consider e.g. $Y=\C^{\times}$.
From the above we conclude that $\Omega^{(\alpha)}_{\C^{\times}}$  is weakly $\Gm$-equivariant
and for $\alpha=0$ strongly $\Gm$-equivariant. The isomorphism 
\begin{align*}
\Omega_{\C^{\times}} \xrightarrow{\cong} \Omega^{(0)}_{\C^{\times}}=\sh{D}_{\C^{\times}}/x\partial_x \sh{D}_{\C^{\times}},\;
\frac{\diff x}{x} \mapsto \overline{1}\;,
\end{align*}
respects the equivariant structures. As the $\Gm$-equivariant irreducible
local systems on $\C^{\times}$ are in bijection with irreducible representations
of the stabilizer $\{\pm 1\}$ of the $\Gm$-action at any point of $\C^{\times}$ \cite{BB81} we in fact also have a
strongly $\Gm$-equivariant structure on $\Omega^{(1/2+n)}_{\C^{\times}}$, $n \in \Z$. 
(The latter is however not induced by the weakly equivariant structure on $\sh{D}_{\C^{\times}}$.)

\subsubsection{Equivariance of direct images}\label{ssec:equivstructureondirectimage}
The direct image $j_{\cdot}\sh{M}$ has an induced
weakly resp. strongly $\Gm$-equivariant structure if 
$\sh{M}$ has a weakly resp. strongly $\Gm$-equivariant structure according to \cite{Kas08}[section 3.5]
resp. \cite{BB81}. By the same token $j_{x!}\Omega_{\C_x}$ and $\iota_{x*}\C$ have
induced strongly $\Gm$-equivariant structures.  
By \autoref{ssec:equivarianceofholonomicdual} we obtain a
weakly or strongly $\Gm$-equivariant structure on $j_! \sh{M}$
and a strongly $\Gm$-equivariant structure on $j_{x!}\Omega_{\C_x}$.
Finally, we also have a weakly or strongly $\Gm$-equivariant structure on $j_{!x\cdot z}\sh{M}$. 

\section{Generalities on $\mathfrak{sl}_2$-modules}\label{sec:generalitiesonsl2modules}
The standard basis of $\mathfrak{sl}_2$ will be denoted
by $e,f,h$. For $\lambda \in \C$ we have the usual $\mathfrak{sl}_2$-module $\MVerma(\lambda)$,
the \emph{Verma module of highest weight $\lambda$}, and $\Lsimple(\lambda)$,
the \emph{simple $\mathfrak{sl}_2$-module of highest weight $\lambda$}. In accordance 
with our convention for $\sh{D}$-modules we will however work with
the corresponding right $\mathfrak{sl}_2$-modules obtained from them via the anti-involution of $\mathfrak{sl}_2$
given by $e \mapsto e$, $f \mapsto f$, $h \mapsto -h$. Thus $\MVerma(\lambda)$ will have $h$-weights
bounded from below, namely $-\lambda + 2\Z_{\geq 0}$. Below, the notation $\Lsimple(\lambda)$ will
only be used when $\Lsimple(\lambda)$ is finite dimensional, i.e. in the case $\lambda \in \Z_{\geq 0}$. 
More generally, we consider \emph{weight modules} for $\mathfrak{sl}_2$. These are
$\mathfrak{sl}_2$-modules $M$ such that $M = \bigoplus_{\lambda \in \C}M_{\lambda}$,
where $M_{\lambda} = \{v \in M\; \vert\; v h = \lambda v\}$ is the \emph{($h$-)weight space
for the weight $\lambda$}. 
Let $T \subseteq \SL_2$ be the diagonal matrices.  

\begin{Rem}\label{Rem:weightmodulesareintegrable}
Let $M$ be a $\mathfrak{sl}_2$-module. Then the $\Lie T = \C h$-action on $M$ integrates to a (algebraic) $T$-action on $M$ if
and only if $M$ is a weight module with weights in $\Z$, see e.g. \cite{Jan03}[I, section 2.11]. 
\end{Rem}

\subsection{Duality $(\cdot)^{\vee}$}\label{ssec:duality}
If $\dim M_{\lambda} < \infty$ for all $\lambda$ we define the \emph{dual of $M$}
 as $M^{\vee}=\bigoplus_{\lambda \in \C}\Hom_{\C}(M_{\lambda},\C)$ by
letting $X \in \mathfrak{sl}_2$ act on $\phi \in M^{\vee}$ by $(\phi X)(v)=\phi(v\tau(X))$. Here $\tau$ is the anti-involution 
of $\mathfrak{sl}_2$ defined by $h \mapsto h$, $e \mapsto f$, $f \mapsto e$. 
Then $M^{\vee}$ is again a weight module and $(M^{\vee})_{\lambda}=\Hom_{\C}(M_{\lambda},\C)$. Thus, if $M$ is e.g. a highest weight module,
then so is $M^{\vee}$. We record for later use that $\MVerma(\lambda) \cong \MVerma(\lambda)^{\vee}$ if and only
if $\lambda \notin \Z_{\geq 0}$. 

\section{$\Ho^0$ and $\Ho^1$}\label{sec:H0andH1}
Let us abbreviate the sheaf cohomology groups $\Ho^k(\Proj^1,\cdot)$ by $\Ho^k(\cdot)$.
Let us from now on drop the subscripts $\Proj^1$ in $\sh{D}_{\Proj^1}$, $\sh{O}_{\Proj^1}$, $\Omega_{\Proj^1}$, etc. to lighten notation.
Consider the $\SL_2$-action on $\Proj^1$ given by $\begin{pmatrix} a & b \\ c & d\end{pmatrix}
[z_0:z_1] = [az_0+bz_1: cz_0+dz_1]$. (The action of the subgroup $T \cong \Gm$ coincides with the one considered
in \autoref{ssec:Gmequivariance}.) It induces an isomorphism of Lie algebras
$\mathfrak{sl}_2 \xrightarrow{\cong} \Ho^0(\sh{T})$,
where $\sh{T}$ is the tangent sheaf of $\Proj^1$, that maps the standard
basis 
\begin{align}\label{eq:vectorfieldssl2}
e \mapsto \partial_x = -z^2\partial_z,\; h \mapsto -2x\partial_x = 2z\partial_z,\; f \mapsto -x^2\partial_x = \partial_z\;.
\end{align} 
($=$ here abbreviates equality after restriction to $\C^{\times}$ and the
global section of $\sh{T}$ is of course obtained by gluing.)
It is known \cite{BB81} that this 
induces an isomorphism $(\Uea\mathfrak{sl}_2)_0 := \Uea\mathfrak{sl}_2/(c) \xrightarrow{\cong} \Ho^0(\sh{D})$ of
algebras. Here $\Uea\mathfrak{sl_2}$ is the universal enveloping algebra of $\mathfrak{sl}_2$ and
$c$ is the Casimir central element $c= ef+fe +\frac{h^2}{2}$. Further, we denote here and
below by $(\cdot)$ the right ideal generated by $\cdot$ in an associative algebra. 
Thus, if $\sh{M}$ is a (right) $\sh{D}$-module, then $\Ho^0(\sh{M})$ 
is a right module over $(\Uea\mathfrak{sl}_2)_0$. 
Recall that we may compute $\Ho^k(\sh{F})$ for any quasicoherent $\sh{O}$-module $\sh{F}$ 
by the Cech complex of an open affine cover \cite{Har77}[III, Theorem 4.5]. For
the cover $\{\C_x,\C_z\}$ the Cech complex is 
\begin{align}\label{eq:Cechcomplex}
\sh{F}(\C_x)\oplus \sh{F}(\C_z) \rightarrow \sh{F}(\C^{\times}),\; (s_1,s_2) \mapsto
s_1\vert \C^{\times} -s_2\vert \C^{\times}\;,
\end{align}
under the usual identification $\C_x\cap \C_z = \C^{\times}$.
If $\sh{F}$ in addition has the structure of a $\sh{D}$-module, then \eqref{eq:Cechcomplex} is $\Ho^0$
of a morphism of $\sh{D}$-modules
\begin{align}\label{eq:sheafCechcomplex}
\mathcal{C}^0(\sh{F}) :=j_{x\cdot}j_x^{-1}\sh{F}\oplus j_{z\cdot}j^{-1}_z\sh{F} \rightarrow \mathcal{C}^1(\sh{F}) :=j_{\cdot}j^{-1}\sh{F}\;.
\end{align}
Hence $\Ho^1(\sh{F})$ is a $\Ho^0(\sh{D})$-module. 
E.g. for $\sh{F}=\Omega$ \eqref{eq:Cechcomplex} is
\begin{align*}
\C[x]\diff x \oplus \C[z]\diff z \rightarrow \C[x,x^{-1}]\diff x,\; 
(x^n\diff x, z^m\diff z) \mapsto (x^n+x^{-m-2})\diff x\;.
\end{align*}
Thus 
\begin{align}\label{eq:H0andH1Omega}
\Ho^0(\Omega)=0\quad \Ho^1(\Omega)=\C x^{-1}\diff x \cong \Lsimple(0)\;.
\end{align}
This calculation can also be found in \cite{Har77}[III, Example 4.0.3]. 
Similarly we find
\begin{align}\label{eq:H0andH1jdot}
& \Ho^0(j_{x\cdot}\Omega_{\C_x})=\Omega_{\C_x}(\C_x)\cong\MVerma(-2) \quad \Ho^1(j_{x\cdot}\Omega_{\C_x})=0\;.
\end{align}
Further, we find with $\iota_{x*}\C \vert \C_z = \sh{D}_{\C_z}/z\sh{D}_{\C_z}$ 
\begin{align}\label{eq:H0andH1istar}
\Ho^0(\iota_{x*}\C)=\C[\overline{\partial_z}] \cong \MVerma(0)\quad \Ho^1(\iota_{x*}\C)=0\;.
\end{align}
Further, we find with $j_{x!}\Omega_{\C_x}\vert \C_z=\sh{D}_{\C_z}/\partial_z z\sh{D}_{\C_z}$, see \eqref{eq:sesforjxshriek},
\begin{align}\label{eq:H0andH1jshriek}
& \Ho^0(j_{x!}\Omega_{\C_x})=\overline{\partial_z} \C[\overline{\partial_z}]\cong \MVerma(0)\quad \Ho^1(j_{x!}\Omega_{\C_x})=\C x^{-1}\diff x\cong \Lsimple(0)\;.
\end{align}

We will denote
by $(\Uea\mathfrak{sl}_2)_0\modulecat$ and $\modulecat(\Uea\mathfrak{sl}_2)_0$ the
category of (not necessarily finitely generated) left resp. right $(\Uea\mathfrak{sl}_2)_0$-modules. 

\subsection{Auto-Equivalences $(\cdot)^-$}\label{ssec:autoequivalences-}
Before proceeding in the computation of
the cohomology it is convenient to write out the operations
induced by exchanging the coordinate $x$ with $z$.
Consider the involution $I$ of $\Proj^1$ given by $x \mapsto x^{-1}$.
If $\sh{M}$ is $\sh{D}$-module, we define $\sh{M}^-$ as the direct image by $I$,
$\sh{M}^- = I_*\sh{M} = I_{\cdot}(\sh{M}\otimes_{\sh{D}}\sh{D}_I)$. Here $\sh{D}_I = \sh{O}
\otimes_{I^{-1}\sh{O}} I^{-1}\sh{D}$
is the transfer $\sh{D}$-$I^{-1}\sh{D}$-bimodule, see e.g. \cite{Kas00}[section 4.1]. Since $I$ is a proper morphism, $(\cdot)^-$ commutes with $\Dual$,
see \cite{Ber}[p. 19].
\begin{Rem}\label{Rem:MotimesRFsheaf}
Let $\sh{R}$ be a sheaf of $\C$-algebras on a topological space $X$. Let $\sh{M}$ be a right $\sh{R}$-module.
Let $\sh{F}$ be a free (not just locally free) left $\sh{R}$-module of rank one. For $U \subseteq X$ open we have isomorphisms of vector spaces
\begin{align*}
\sh{M}(U)\otimes_{\sh{R}(U)}\sh{F}(U) \cong \sh{M}(U)
\end{align*}
compatible with the restriction maps. Thus, the sheafification maps $\sh{M}(U)\otimes_{\sh{R}(U)}\sh{F}(U) \rightarrow (\sh{M}\otimes_{\sh{R}}\sh{F})(U)$, which we have by definition of $\sh{M}\otimes_{\sh{R}}\sh{F}$,
are isomorphisms. 
\end{Rem}

\begin{Lemma}\label{Lemma:Itransferbimodule}
Let $U \subseteq \Proj^1$ be open. The left action of $\sh{D}$ on $\sh{D}_I$ induces an isomorphism
\begin{align*}
l:\sh{D}(U) \xrightarrow{\cong} \sh{O}(U)\otimes_{(I^{-1}\sh{O})(U)} (I^{-1}\sh{D})(U) \xrightarrow{\cong} \sh{D}_I(U)\;,
\end{align*}
where the first map is $p \mapsto p(1 \otimes 1)$.  
The right action of $I^{-1}\sh{D}$ on $\sh{D}_I$ induces
an isomorphism
\begin{align*}
r: \sh{D}(I(U)) \xrightarrow{\cong}  \sh{O}(U)\otimes_{(I^{-1}\sh{O})(U)} (I^{-1}\sh{D})(U) \xrightarrow{\cong} \sh{D}_I(U)\;,
\end{align*}
where the first map is $p \mapsto 1\otimes p$. The compositions $r^{-1}l$ define an isomorphism of sheaves
of algebras $J: \sh{D} \xrightarrow{\cong} I^{-1}\sh{D}$ such that $(I^{-1}J)J=\id_{\sh{D}}$.
\end{Lemma}
\begin{proof}
Remark \autoref{Rem:MotimesRFsheaf} shows that the second map in the definition of $l$ and $r$ is an isomorphism
since $I^{\sharp}: I^{-1}\sh{O} \rightarrow \sh{O}$ is an isomorphism of sheaves of algebras. 
For the remaining statements one recalls \cite{Kas00}[section 4.1] the definition of the left action
of $\sh{D}$ on $\sh{D}_I$. The differential of $I$ induces an isomorphism of sheaves of Lie algebras $\sh{T} \xrightarrow{\cong} I^{-1}\sh{T}$.
This map together with $(I^{\sharp})^{-1}: \sh{O} \xrightarrow{\cong} I^{-1}\sh{O}$ induces an isomorphism of sheaves of algebras $\sh{D} \xrightarrow{\cong} I^{-1}\sh{D}$, which is $J$.
\end{proof}

\begin{Rem}\label{Rem:involutiononH0ofTangent}
From \eqref{eq:vectorfieldssl2} we see that under the isomorphism
$\mathfrak{sl}_2 \xrightarrow{\cong} \Ho^0(\sh{T})$
the map $\sh{T} \rightarrow I^{-1}\sh{T}$ of the proof of Lemma \autoref{Lemma:Itransferbimodule} 
is given by $h \mapsto -h$, $e \mapsto f$, $f \mapsto e$. 
\end{Rem}
$J$ induces an involution of $(\Uea\mathfrak{sl}_2)_0$. 
We will denote the twist of a $(\Uea\mathfrak{sl}_2)_0$-module
$M$ by this involution by $M^-$. 
If $M$ is a lowest weight module, then by remark \autoref{Rem:involutiononH0ofTangent} $M^-$ is a highest weight
module and vice versa. $(\cdot)^-$ defines an exact auto-equivalence of the category
$\modulecat\sh{D}$ and $\modulecat (\Uea\mathfrak{sl}_2)_0$ respectively. 
We have $\tau J = J \tau$ and hence $(\cdot)^-$ commutes with $(\cdot)^{\vee}$.
Here $\tau$ and $(\cdot)^{\vee}$ were defined in \autoref{ssec:duality}. 

\begin{Lemma}\label{Lemma:autoequivalenceintertwinedbyHo}
There are natural isomorphisms $\Ho^k(\sh{M}^-)\cong\Ho^k(\sh{M})^-$
for any $\sh{M} \in \modulecat\sh{D}$.
\end{Lemma}
\begin{proof}
We have $\Ho^k(I_*\sh{M}) \cong \Ho^k(\sh{M}\otimes_{\sh{D}}\sh{D}_I)$
since $I$ is an affine morphism.  
We reason with the complex \eqref{eq:sheafCechcomplex}.
We have an isomorphism $\Ho^0(\sh{C}^{\bullet}(\sh{M}))\otimes_{\Ho^0(\sh{D})}\Ho^0(\sh{D}_I) \xrightarrow{\cong}
\Ho^0(\sh{C}^{\bullet}(\sh{M})\otimes_{\sh{D}}\sh{D}_I)$ of complexes
due to remark \autoref{Rem:MotimesRFsheaf}. 
Further, we claim that we have
an isomorphism of complexes $\Ho^0(\sh{C}^{\bullet}(\sh{M})\otimes_{\sh{D}}\sh{D}_I) \cong \Ho^0(\sh{C}^{\bullet}(\sh{M}\otimes_{\sh{D}}\sh{D}_I))$. Indeed, this 
follows from the fact that for any open embedding $\kappa$ into $\Proj^1$ we have
$(\kappa_{\cdot}\kappa^{-1}\sh{M})\otimes_{\sh{D}}\sh{D}_I \cong \kappa_{\cdot}\kappa^{-1}(\sh{M}\otimes_{\sh{D}}\sh{D}_I)$
as $\sh{D}$-modules, where $\sh{D}$ acts via $\sh{D} \rightarrow \kappa_{\cdot}\kappa^{-1}\sh{D}$,
as can be checked on open subsets. 
In this way we obtain an isomorphism of complexes $\Ho^0(\sh{C}^{\bullet}(\sh{M}))\otimes_{\Ho^0(\sh{D})}\Ho^0(\sh{D}_I) \xrightarrow{\cong}
\Ho^0(\sh{C}^{\bullet}(\sh{M}\otimes_{\sh{D}}\sh{D}_I))$. 
Lemma \autoref{Lemma:Itransferbimodule} together with the definition of $(\cdot)^-$ now implies the statement.
\end{proof}

As an application of this lemma we e.g. deduce from $(j_{x\cdot}\Omega_{\C_x})^- \cong j_{z\cdot}\Omega_{\C_z}$ and \eqref{eq:H0andH1jdot} that
\begin{align*}
& \Ho^0(j_{z\cdot}\Omega_{\C_z})=\Omega_{\C_z}(\C_z)\cong\MVerma(-2)^- \quad \Ho^1(j_{z\cdot}\Omega_{\C_z})=0\;,
\end{align*}
but of course this can also be computed directly. 

\subsection{$\Ho^0$ and $\Ho^1$ of $j_{\cdot}\Omega^{(\alpha)}_{\C^{\times}}$, $\alpha \notin \Z$,
and definition of $\RVerma(\alpha)$}

According to \eqref{eq:jdotOmegaalphaonCx} the restriction $(j_{\cdot}\Omega^{(\alpha)}_{\C^{\times}})(\C_x)
\rightarrow (j_{\cdot}\Omega^{(\alpha)}_{\C^{\times}})(\C^{\times})$ is an isomorphism of vector spaces. Thus $\Ho^1(j_{\cdot}\Omega^{(\alpha)}_{\C^{\times}})=0$. Further, according to \eqref{eq:jdotOmegaalphaonCx} 
and the corresponding statement for $\C_z$ 
\begin{align*}
\Ho^0(j_{\cdot}\Omega^{(\alpha)}_{\C^{\times}}) = \Ho^0(\C^{\times},\Omega^{(\alpha)}_{\C^{\times}}) \cong (\Uea \mathfrak{sl}_2)_0/(h+2\alpha)\;.
\end{align*}
We set
\begin{align}\label{eq:defofRVerma}
\RVerma(\alpha) = (\Uea \mathfrak{sl}_2)_0/(h+2\alpha) = \Uea\mathfrak{sl}_2/(h+2\alpha,\; ef+\alpha(\alpha+1))\;.
\end{align}
The PBW theorem for $\Uea\mathfrak{sl}_2$ implies $\RVerma(\alpha)=\overline{f}\C[\overline{f}]\oplus \C[\overline{e}]$ as a vector space. Due to remark \autoref{Rem:eltarypropertiesofOmegaalpha} (2)
and the fact that the $h$-weights of $\RVerma(\alpha)$ 
are $-2\alpha+2\Z$ we have $\RVerma(\alpha) \cong \RVerma(\beta)$ if and only if $\alpha-\beta \in \Z$.  
The following lemma and its proof is analogous to Lemma \autoref{Lemma:alphanotanintegerjdotsimple}.
\begin{Lemma}\label{Lemma:RVermaalphasimple}
$\RVerma(\alpha)$ is a simple $\mathfrak{sl}_2$-module for $\alpha \notin \Z$.
\end{Lemma}
\begin{proof}
Let $M \neq 0$ be a submodule of $\RVerma(\alpha)$. As in the proof of Lemma \autoref{Lemma:alphanotanintegerjdotsimple}
we show that $M$ is a direct sum of its $h$-weight spaces $M = \bigoplus_{n \in \Z}M_{-2\alpha+2n}$
with $M_{-2\alpha+2n}\neq 0$ for some $n$. We have linear maps
\begin{align*}
\xymatrix{ M_{-2\alpha+2(n+1)} \ar@<5pt>[r]^{\cdot e}
& \ar[l]^{\cdot f} M_{-2\alpha+2n}}\;.
\end{align*}
From the definition \eqref{eq:defofRVerma} it follows
\begin{align}\label{eq:efandfeonweightspaces}
\cdot ef =(n-\alpha)(\alpha-n-1)\id_{M_{-2\alpha+2(n+1)}}
\qquad \cdot fe = (n-\alpha)(\alpha-n-1)\id_{M_{-2\alpha+2n}}\;.
\end{align}
Thus, since $\alpha \notin \Z$, the maps $e$ and $f$ are invertible for each $n \in \Z$. It follows $M_{-2\alpha+2n} \neq 0$
for all $n \in \Z$ and hence $M = \RVerma(\alpha)$. 
\end{proof}

\subsection{$\Ho^0$ and $\Ho^1$ of $j_{\cdot}\Omega_{\C^{\times}}$}\label{ssec:Ho0andHo1ofjdot}
We find $\Ho^0(j_{\cdot}\Omega_{\C^{\times}})=\Omega_{\C^{\times}}(\C^{\times})\cong \RVerma(0)$, now allowing
$\alpha = 0$ in the definition \eqref{eq:defofRVerma}. Also $\Ho^1(j_{\cdot}\Omega_{\C^{\times}})=0$.
\begin{Rem}
For any $\alpha \in \C$ we have $(j_{\cdot}\Omega^{(\alpha)}_{\C^{\times}})^- \cong j_{\cdot}\Omega^{(-\alpha)}_{\C^{\times}}$
and by Remark \autoref{Rem:involutiononH0ofTangent} $\RVerma(\alpha)^-\cong\RVerma(-\alpha)$.
\end{Rem}

Since $\Ho^1(j_{x\cdot}\Omega_{\C_x})=0$ the long exact sequence of cohomology of \eqref{eq:sesforjdotOmegaCtimes} gives the exact sequence
\begin{align*}
0 \rightarrow \MVerma(-2) \rightarrow \RVerma(0) \rightarrow \MVerma(0)^- \rightarrow 0\;.
\end{align*}
Applying $(\cdot)^-$ we obtain the exact sequence
\begin{align*}
0 \rightarrow \MVerma(-2)^- \rightarrow \RVerma(0) \rightarrow \MVerma(0) \rightarrow 0\;.
\end{align*}

\subsection{$\Ho^0$ and $\Ho^1$ of $j_!\Omega_{\C^{\times}}$}\label{ssec:Ho0andHo1ofj!}
From $j_!\Omega_{\C^{\times}}\vert \C_z \cong j_{x!}\Omega_{\C_x}\vert \C_z \cong \sh{D}_{\C_z}/\partial_z z\sh{D}_{\C_z}$, see \eqref{eq:sesforjxshriek},
and the corresponding statement for $\C_x$ we deduce $\Ho^0(j_!\Omega_{\C^{\times}}) \cong \MVerma(0)\oplus \MVerma(0)^-$.
Also $\Ho^1(j_!\Omega_{\C^{\times}})=x^{-1}\diff x \cong \Lsimple(0)$. Since $\Ho^1(\iota_{z*}\C)=0$ the
long exact sequence of cohomology of \eqref{eq:sesforjshriekOmegaCtimes} gives the exact sequence
\begin{align*}
0 \leftarrow \MVerma(0) \leftarrow \MVerma(0)\oplus\MVerma(0)^- \leftarrow \MVerma(0)^- \leftarrow 0\;.
\end{align*}

\subsection{$\Ho^0$ and $\Ho^1$ of $j_{!x\cdot z}\Omega_{\C^{\times}}$ and $j_{\cdot x !z}\Omega_{\C^{\times}}$}
As in \autoref{ssec:Ho0andHo1ofjdot} and \autoref{ssec:Ho0andHo1ofj!} we find
$\Ho^0(j_{!x\cdot z}\Omega_{\C^{\times}})\cong \RVerma(-1)$ and $\Ho^1(j_{!x\cdot z}\Omega_{\C^{\times}})=0$
as well as $\Ho^0(j_{\cdot x !z}\Omega_{\C^{\times}})\cong \RVerma(1)$ and $\Ho^1(j_{\cdot x !z}\Omega_{\C^{\times}})=0$.
Again, one can write down the long exact sequence of cohomology of \eqref{eq:sesforjdotx!zOmegaCtimes}.

\begin{Rem}\label{Rem:BBuntwisted}
We compare the above results for $\Ho^0$ and $\Ho^1$
with the equivalence of Beilinson-Bernstein for \emph{left} $\sh{D}$-modules
\cite{BB81}:
Then $\Ho^1 = 0$ and $\Ho^0$ is an exact equivalence of categories $\sh{D}\modulecat \rightarrow (\Uea\mathfrak{sl}_2)_0\modulecat$.
The functor $\Delta$:  $(\Uea\mathfrak{sl}_2)_0\modulecat \rightarrow$ $\sh{D}\modulecat$
\cite{BB81} defined by $\Delta(M) = \sh{D} \otimes_{\underline{(\Uea\mathfrak{sl}_2)_0}} \underline{M}$,
where $\underline{V}$ denotes the constant sheaf on $\Proj^1$ constructed from the vector space $V$, is quasi-inverse
to $\Ho^0$. Namely, the natural morphism of $\sh{D}$-modules 
\begin{align}\label{eq:DmodulegeneratedfromH0}
\sh{D} \otimes_{\underline{(\Uea\mathfrak{sl}_2)_0}} \underline{\Ho^0(\sh{M})} 
\rightarrow \sh{M}
\end{align}
is an isomorphism of $\sh{D}$-modules. Note that a $\sh{D}$-module is generated by its global sections if and only if
\eqref{eq:DmodulegeneratedfromH0} surjects. Of the above $\sh{D}$-modules,
only $j_{x\cdot}\Omega_{\C_x}$ and $\iota_{x*}\C$ satisfy this. In fact, $j_{x\cdot}\Omega_{\C_x}$ is generated
by its global sections even as an $\sh{O}$-module. 
In the above $\Ho^0$ and $\Ho^1$ the dual Verma module $\MVerma(0)^{\vee}$ does not appear
and $\MVerma(0)$ is $\Ho^0$ both of $\iota_{x*}\C$ and $j_{x!}\Omega_{\C_x}$.
In fact, by \cite{BB93}[3.3.3 Corollary] the equivalence $\Ho^0: \sh{D}\modulecat \rightarrow (\Uea\mathfrak{sl}_2)_0\modulecat$ 
restricts to an equivalence between the subcategories of strongly $\Gm$-equivariant $\sh{D}$-modules
and of $(\Uea\mathfrak{sl}_2)_0$-modules on which the $\C h$-action comes from a $T$-action. 
Coming back to our computations of $\Ho^0$, it is manifest, cf. Remark \autoref{Rem:weightmodulesareintegrable}, that
the direct images which are strongly $\Gm$-equivariant as discussed in \autoref{ssec:equivstructureondirectimage}
have the property that their $\Ho^0$ carry a $\C h$-action coming from a $T$-action.
\end{Rem}

\section{Twisted $\sh{D}$-modules}\label{sec:twistedDmodules}
\subsection{Generalities on twisted $\sh{D}$-modules}\label{ssec:generalitiesontwistedDmodules}
Let $\lambda \in \Z$. We have a corresponding line bundle $\sh{O}(\lambda)$
on $\Proj^1$. The sheaf of algebras
$\sh{D}(\lambda) :=\sh{O}(-\lambda)\otimes_{\sh{O}}\sh{D}\otimes_{\sh{O}} \sh{O}(\lambda)$ is the sheaf of $\sh{O}(\lambda)$-twisted differential operators \cite{BB81}, \cite{BB93}[section 2.1]. E.g. $\Omega \otimes_{\sh{O}} \sh{O}(\lambda)$ is naturally a (right) $\sh{D}(\lambda)$-module. 
$\sh{D}(\lambda)$ possesses an increasing filtration $F^n \sh{D}(\lambda)$, $n \in \Z_{\geq 0}$, by
locally free $\sh{O}$-submodules of finite rank defined by the order of the differential operator. 
Let us first recall a description of $F^1\sh{D}(\lambda)$ given in \cite{BB93}[section 2.1]. 

\begin{Rem} The commutator $[\cdot,\cdot]$ in $\sh{D}(\lambda)$
turns $F^1\sh{D}(\lambda)$ into a sheaf of Lie algebras. There is a short exact sequence of locally free $\sh{O}$-modules
\begin{align}\label{eq:LiealgebroidF1Dlambda}
0 \rightarrow F^0\sh{D}(\lambda)=\sh{O} \rightarrow F^1\sh{D}(\lambda) \rightarrow \sh{T} \rightarrow 0\;,
\end{align}
where $F^1\sh{D}(\lambda) \rightarrow \sh{T}=\sh{D}er(\sh{O})$, $p \mapsto [p,\cdot]$.
Here $\sh{D}er(\sh{O})$ is the $\sh{O}$-module of derivations. This makes $F^1\sh{D}(\lambda)$
into a Picard Lie algebroid in the sense of \cite{BB93}[section 2.1], a special kind of Lie algebroid. 
\end{Rem}

\begin{Rem}
$F^1 \sh{D}(-\lambda)$ identifies with the Atiyah algebroid of the principal $\Gm$-bundle $\sh{O}(\lambda)^{\times}$
associated to $\sh{O}(\lambda)$.
Let us elaborate on this. The \emph{Atiyah algebroid} of $\sh{O}(\lambda)^{\times}$ is given by an exact sequence
\begin{align}\label{eq:Atiyahalgebroid}
0 \rightarrow \Lie \Gm \times_{\Gm} \sh{O}(\lambda)^{\times} \rightarrow \left(p_{\cdot} \sh{T}_{\sh{O}(\lambda)^{\times}}\right)^{\Gm} \rightarrow
\sh{T} \rightarrow 0\;,
\end{align}
where $p: \sh{O}(\lambda)^{\times} \rightarrow \Proj^1$ denotes the projection and $(\cdot)^{\Gm}$ $\Gm$-invariants. 
The first map in \eqref{eq:Atiyahalgebroid} realizes 
the adjoint bundle $\Lie \Gm \times_{\Gm}\sh{O}(\lambda)^{\times}\cong \sh{O}$ as the vertical vector fields on 
the total space of the $\Gm$-bundle, which is also denoted by $\sh{O}(\lambda)^{\times}$. We have an isomorphism  $p_{\cdot} \sh{O}_{\sh{O}(\lambda)^{\times}}
\cong \bigoplus_{n \in \Z}\sh{O}(\lambda n)$ of $\Z$-graded sheaves of $\sh{O}_{\Proj^1}$-algebras. Here
the homogeneous component
$\sh{O}(\lambda n)$ identifies with the subsheaf on which the vertical vector fields act by multiplication by $n$. 
$\left(p_{\cdot} \sh{T}_{\sh{O}(\lambda)^{\times}}\right)^{\Gm}$ acts on $p_{\cdot}\sh{O}_{\sh{O}(\lambda)^{\times}}$
preserving the homogeneous components. In this way we obtain a map $\left(p_{\cdot} \sh{T}_{\sh{O}(\lambda)^{\times}}\right)^{\Gm} \rightarrow F^1\sh{D}(-\lambda)$, which turns out
to be an isomorphism of Lie algebroids. 
\end{Rem}

We have a map of Lie algebroids 
\begin{align}\label{eq:algebroidac}
\ac: \mathfrak{sl}_2\otimes_{\C}\sh{O}_{\Proj^1} \rightarrow F^1\sh{D}(\lambda)
\end{align}
induced by the natural $\SL_2$-equivariant structure of $\sh{O}(\lambda)$. 
This makes precise the statement that $\ac$ lifts the infinitesimal action of $\SL_2$ on $\Proj^1$
to first order differential operators.
The expression $z_0^{\lambda}$ respectively $z_1^{\lambda}$ is a nowhere vanishing section of $\sh{O}(\lambda)$ on $\C_x$
respectively $\C_z$. On global sections we have
\begin{align}\label{eq:twistedfirstorder}
\begin{split}
\ac(e\otimes 1) &= z_0^{-\lambda}\otimes\partial_x\otimes z_0^{\lambda} = z_1^{-\lambda}\otimes(-z^2\partial_z-\lambda z)\otimes z_1^{\lambda}\\
\ac(h\otimes 1) &= z_0^{-\lambda}\otimes(-2x\partial_x-\lambda)\otimes z_0^{\lambda} = z_1^{-\lambda}\otimes(2z\partial_z+\lambda)\otimes z_1^{\lambda}\\
\ac(f\otimes 1) &= z_0^{-\lambda}\otimes(-x^2\partial_x-\lambda x)\otimes z_0^{\lambda} = z_1^{-\lambda} \otimes \partial_z \otimes z_1^{\lambda}\;, 
\end{split}
\end{align}
cf. \eqref{eq:vectorfieldssl2}. 
\begin{Rem} The space of extensions of the form \eqref{eq:LiealgebroidF1Dlambda} is $\Exten^1(\sh{T},\sh{O})=\Ho^1(\Omega^1) = \C \frac{\diff x}{x}$, see \eqref{eq:H0andH1Omega}. It has the following Cech description. Let $\sigma_{x,z}: \sh{T}(\C_{x,z}) \rightarrow F^1\sh{D}(\lambda)(\C_{x,z})$
be $\sh{O}(\C_{x,z})$-linear sections of \eqref{eq:LiealgebroidF1Dlambda} over $\C_{x,z}$. If $\sigma_x^{\prime}$ is 
another such section, then $\sigma_x-\sigma_x^{\prime}: \sh{T}(\C_x) \rightarrow \sh{O}(\C_x)$ is an arbitrary
$\sh{O}(\C_x)$-linear map, i.e. an element of $\Omega^1(\C_x)$. The difference $\sigma_x \vert \C^{\times} 
-\sigma_z \vert \C^{\times}: \sh{T}(\C^{\times}) \rightarrow \sh{O}(\C^{\times})$ is $\sh{O}(\C^{\times})$-linear
and hence defines an element of $\Omega^1(\C^{\times})$. The extension class is given by the class
of $\sigma_x \vert \C^{\times} -\sigma_z \vert \C^{\times}$ in $\Ho^1(\Omega^1)$, which is independent of
the choice of $\sigma_{x,z}$. Let us compute the extension class from the formulae \eqref{eq:twistedfirstorder}.
The assignment $\partial_x \mapsto z_0^{-\lambda}\otimes \partial_x \otimes z_0^{\lambda}$
defines a unique $\sigma_x$ and  $\partial_z \mapsto z_1^{-\lambda}\otimes \partial_z \otimes z_1^{\lambda}$ defines
a unique $\sigma_z$. By $\sh{O}(\C_z)$-linearity $\sigma_z$ sends $-z^2\partial_z \mapsto z_1^{-\lambda}\otimes(-z^2\partial_z)\otimes z_1^{\lambda}$ and consequently by the first line in \eqref{eq:twistedfirstorder} $\sigma_x\vert \C^{\times} - \sigma_z\vert \C^{\times}$
sends 
\begin{align*}
\partial_x = -z^2\partial_z \mapsto z_1^{-\lambda}\otimes(-\lambda z)\otimes z_1^{\lambda} = -\lambda z\;.
\end{align*} 
Thus, the extension class is $-\lambda \frac{\diff x}{x}$, which is the first Chern class of $\sh{O}(-\lambda)$. 
This is a general fact, cf. \cite{BB93}[section 2.1]. 
\end{Rem}

\eqref{eq:algebroidac} induces a map of sheaves of algebras
$\Uea(\mathfrak{sl}_2\otimes_{\C}\sh{O}_{\Proj^1}) \rightarrow \sh{D}(\lambda)$, where $\Uea(\cdot)$
is defined in \cite{BB93}[section 1.2.5]. 
The corresponding map on global sections $\Uea\mathfrak{sl}_2 \rightarrow \Ho^0(\sh{D}(\lambda))$
induces an isomorphism $(\Uea\mathfrak{sl}_2)_{\chi_{\lambda}} := \Uea\mathfrak{sl}_2/(c -\frac{1}{2}\lambda(\lambda-2)) \xrightarrow{\cong} \Ho^0(\sh{D}(\lambda))$ \cite{BB81}.
Thus, if $\sh{M}$ is a $\sh{D}(\lambda)$-module, then the $\Ho^j(\sh{M})$ are naturally $(\Uea\mathfrak{sl}_2)_{\chi_{\lambda}}$-modules. 

\subsection{$\Ho^0$ and $\Ho^1$} The results of this section
generalize the ones of \autoref{sec:H0andH1}. 
From the Cech complex \eqref{eq:Cechcomplex}  and \eqref{eq:twistedfirstorder} we find
\begin{Thm}\label{Thm:hwmodulestwisted}
\begin{align}\label{eq:twistedjxextensions}
\begin{split}
& \Ho^0(\Omega\otimes_{\sh{O}} \sh{O}(\lambda))\cong \begin{cases} \Lsimple(\lambda-2) & \lambda \geq 2 \\ 0 & \lambda \leq 1\end{cases}
\quad \Ho^1(\Omega \otimes_{\sh{O}}\sh{O}(\lambda))\cong \begin{cases} 0 & \lambda \geq 1 \\ \Lsimple(-\lambda) & \lambda \leq 0\end{cases}\\
& \Ho^0(j_{x\cdot}\Omega_{\C_x})\cong \MVerma(\lambda-2)^{\vee} \quad \Ho^1(j_{x\cdot}\Omega_{\C_x})=0\\
& \Ho^0(\iota_{x*}\C) \cong \MVerma(-\lambda)\quad \Ho^1(\iota_{x*}\C) = 0\\
& \Ho^0(j_{x!}\Omega_{\C_x}) \cong \begin{cases} \MVerma(\lambda-2) & \lambda \geq 1 \\ \MVerma(-\lambda) & \lambda \leq 0\end{cases}\quad \Ho^1(j_{x!}\Omega_{\C_x})\cong\begin{cases} 0 & \lambda \geq 1 \\ \Lsimple(-\lambda) & \lambda \leq 0\end{cases}
\end{split}\;.
\qed
\end{align} 
\end{Thm}
Even though this is not reflected in the notation, the direct images $j_{x\cdot}\Omega_{\C_x}$ etc.  
are now of course understood as $\sh{D}(\lambda)$-modules, differing from the previously considered modules
by a factor $\sh{O}(\lambda)$. 
Theorem \autoref{Thm:hwmodulestwisted} is certainly well-known. In fact
a similar statement is known to hold for any semisimple Lie algebra over $\C$ instead of $\mathfrak{sl}_2$.
The first part is the Borel-Weil-Bott theorem, while the identification of the cohomology of $j_{x\cdot}\Omega_{\C_x}$ and $\iota_{x*}\C$ is proven in \cite{BK81}[equation (5.1.2) and Corollary 5.8] and \cite{Gai05}[Theorem 10.6] in the untwisted case.

\begin{Rem}
Independent of the twist it holds that $\Ho^0$ of $\iota_{x*}$ resp. $j_{x\cdot}$ is a Verma resp. dual Verma module. 
\end{Rem}
Let $\modulecat\sh{D}(\lambda)$ denote the category of $\sh{D}(\lambda)$-modules quasicoherent over
$\sh{O}$. The definitions and results of \autoref{ssec:autoequivalences-}
carry over to the twisted case: Since $I^*\sh{O}(\lambda) \cong \sh{O}(\lambda)$ we again have
an exact auto-equivalence $(\cdot)^-$ of $\modulecat\sh{D}(\lambda)$, further an exact auto-equivalence
$(\cdot)^-$ of $\modulecat(\Uea\mathfrak{sl}_2)_{\chi_{\lambda}}$ and Lemma \autoref{Lemma:autoequivalenceintertwinedbyHo}. 

\subsubsection{$\Ho^0$ and $\Ho^1$ of $j_{\cdot}\Omega^{(\alpha)}_{\C^{\times}}$, $\alpha \notin \Z$,
and definition of $\RVerma(\lambda,\alpha)$}
We find
\begin{Thm}\label{Thm:jdotOmegaalpha}
\begin{align*}
\Ho^0(j_{\cdot}\Omega^{(\alpha)}_{\C^{\times}})\cong (\Uea\mathfrak{sl}_2)_{\chi_{\lambda}}/(h+2\alpha+\lambda)\quad \Ho^1(j_{\cdot}\Omega^{(\alpha)}_{\C^{\times}})=0\;.
\end{align*}\qed
\end{Thm}
We set
\begin{align}\label{eq:RVermaalphalambda}
\RVerma(\lambda,\alpha) = (\Uea \mathfrak{sl}_2)_{\chi_{\lambda}}/(h+2\alpha+\lambda) = \Uea\mathfrak{sl}_2/(h+2\alpha+\lambda, ef+(\alpha+\lambda)(\alpha+1))
\end{align}
generalizing \eqref{eq:defofRVerma}. The following
lemma and its proof generalize Lemma \autoref{Lemma:RVermaalphasimple}
and hence we will be brief.
\begin{Lemma} Let $\alpha \notin \Z$. 
$\RVerma(\lambda,\alpha)$ is simple.
$\RVerma(\lambda,\alpha) \cong \RVerma(\mu,\beta)$ if and only if
$\mu \in \{\lambda,-\lambda+2\}$ and $\alpha-\beta \in \Z$.
\end{Lemma}
\begin{proof}
We conclude from \eqref{eq:RVermaalphalambda} for any submodule $M$ of $\RVerma(\lambda,\alpha)$
\begin{align}\label{eq:twistedefandfeonweightspaces}
\begin{split}
\cdot ef &=(n-\alpha)(\alpha-n-1+\lambda)\id_{M_{-2\alpha-\lambda+2(n+1)}}\\
\cdot fe &= (n-\alpha)(\alpha-n-1+\lambda)\id_{M_{-2\alpha-\lambda+2n}}
\end{split}
\end{align}
generalizing \eqref{eq:efandfeonweightspaces}.  $\RVerma(\lambda,\alpha)\cong \RVerma(\mu,\beta)$
implies that the modules have the same central character. Thus $\mu \in \{\lambda,-\lambda+2\}$.
As the weights of $\RVerma(\lambda,\alpha)$ are $-2\alpha-\lambda + 2\Z$ we conclude $\alpha -\beta \in \Z$.  
\end{proof}

\subsubsection{Description of $\RVerma(\lambda,\alpha)$, $\alpha \in \Z$}\label{ssec:descriptionofRVermalambdaalpha}
Before proceeding to the computation of the remaining $\Ho^0$ and $\Ho^1$, it is convenient to analyze $\RVerma(\lambda,\alpha)$
for $\alpha \in \Z$. We will find that for many values of the parameter $\alpha$ the 
$\RVerma(\lambda,\alpha)$ are isomorphic. Hence, we will introduce some notation for the
isomorphism classes.

\begin{Rem}\label{Rem:-onRVermalambdaalpha}
For any $\alpha \in \C$ $\RVerma(\lambda,\alpha)^- \cong \RVerma(\lambda,-\alpha-\lambda)$ follows from the definition
\eqref{eq:RVermaalphalambda}.
\end{Rem}

\begin{Lemma}\label{Lemma:RVermalambdaalpha}
Let $\alpha,\beta \in \Z$. 
\begin{enumerate}
\item \emph{Case $\lambda \geq 2$.}  There are three isomorphism classes of $\RVerma(\lambda,\alpha)$
given by $\alpha \leq -\lambda$, $\alpha \geq 0$ and $1-\lambda \leq \alpha \leq -1$ respectively. 
We denote them by $\RVerma(\lambda,<)$, $\RVerma(\lambda,>)$ and $\RVerma(\lambda,=)$.
We have exact sequences
\begin{align}
\label{eq:first<} & 0 \rightarrow \MVerma(-\lambda)^- \rightarrow \RVerma(\lambda,<) \rightarrow \MVerma(\lambda-2)^{\vee} \rightarrow 0\\
\label{eq:second<} & 0 \rightarrow \MVerma(\lambda-2)^- \rightarrow \RVerma(\lambda,<) \rightarrow \MVerma(-\lambda) \rightarrow 0\\
\label{eq:first>} & 0 \rightarrow \MVerma(-\lambda) \rightarrow \RVerma(\lambda,>) \rightarrow \MVerma(\lambda-2)^{-\vee} \rightarrow 0\\
\label{eq:second>} & 0 \rightarrow \MVerma(\lambda-2) \rightarrow \RVerma(\lambda,>) \rightarrow \MVerma(-\lambda)^- \rightarrow 0\\
\label{eq:first=} & 0 \rightarrow \MVerma(-\lambda) \rightarrow \RVerma(\lambda,=) \rightarrow \MVerma(\lambda-2)^- \rightarrow 0\\
\label{eq:second=} & 0 \rightarrow \MVerma(-\lambda)^- \rightarrow \RVerma(\lambda,=) \rightarrow \MVerma(\lambda-2) \rightarrow 0\;.
\end{align}

\item \emph{Case $\lambda \leq 0$.}
There are three isomorphism classes of $\RVerma(\lambda,\alpha)$ given by $\alpha \leq -1$, $\alpha \geq 1-\lambda$ and $0 \leq \alpha \leq -\lambda$
respectively. We denote them by $\RVerma(\lambda,<)$, $\RVerma(\lambda,>)$ and $\RVerma(\lambda,=)$.
Of course we have similar exact sequences as in (1).

\item There are two isomorphism classes of $\RVerma(1,\alpha)$ given by $\alpha \leq -1$ and $\alpha \geq 0$ respectively.
We denote them by $\RVerma(1,<)$ and $\RVerma(1,>)$.
We have exact sequences
\begin{align}
\label{eq:1<} & 0 \rightarrow \MVerma(-1)^- \rightarrow \RVerma(1,<) \rightarrow \MVerma(-1) \rightarrow 0\\
\label{eq:1>} & 0 \rightarrow \MVerma(-1) \rightarrow \RVerma(1,>) \rightarrow \MVerma(-1)^- \rightarrow 0\;.
\end{align}

\item Let $\lambda \neq \mu$. $\RVerma(\lambda,\alpha)\cong \RVerma(\mu,\beta)$
if and only if $\mu=-\lambda+2$ and $\alpha$ and $\beta$ belong to corresponding
isomorphism classes in (1) and (2). 

\item We have $\RVerma(\lambda,<)^- \cong \RVerma(\lambda,<)^{\vee} \cong  \RVerma(\lambda,>)$
and $\RVerma(\lambda,=)^- \cong \RVerma(\lambda,=)$. $(\cdot)^-$ exchanges \eqref{eq:first<} with \eqref{eq:first>}, \eqref{eq:second<} with \eqref{eq:second>}, \eqref{eq:first=} with \eqref{eq:second=} and \eqref{eq:1<} with \eqref{eq:1>}. 
$(\cdot)^{\vee}$ exchanges \eqref{eq:first<} with \eqref{eq:second>}, \eqref{eq:second<} with \eqref{eq:first>}
and \eqref{eq:1<} with \eqref{eq:1>}.

\end{enumerate}
\end{Lemma}
\begin{proof}
(1). From \eqref{eq:twistedefandfeonweightspaces} we deduce $\cdot ef = 0$
and $\cdot fe = 0$ if and only if $n=\alpha$ or $n=\alpha-1+\lambda$. Thus,
for $\nu \notin \{\lambda,-\lambda+2\}$ the arrows
\begin{align*}
\xymatrix{ \RVerma(\lambda,\alpha)_{\nu} \ar@<5pt>[r]^{\cdot e}
& \ar[l]^{\cdot f} \RVerma(\lambda,\alpha)_{\nu-2}}
\end{align*}
are invertible and for $\nu \in \{\lambda,-\lambda+2\}$ exactly
one of the arrows is zero (they cannot both be zero as $\RVerma(\lambda,\alpha)$ is 
not a direct sum). Further, $f=0$ for $\nu = \lambda$
together with $e=0$ for $\nu = -\lambda+2$ cannot occur since $\RVerma(\lambda,\alpha)$ is
cyclic.  We are left with three cases. In each case the allowed range for $\alpha$
follows from the fact that the generator $\overline{1}$ of $\RVerma(\lambda,\alpha)$ 
has weight $-2\alpha-\lambda$.\newline
(2) and (3). See the proof of (1).\newline
(4) $\RVerma(\lambda,\alpha)\cong \RVerma(\mu,\beta)$ implies that the modules have the
same central character. Thus $\mu =-\lambda+2$. From the definition \eqref{eq:RVermaalphalambda}
we see $\RVerma(\lambda,\alpha) \cong \RVerma(-\lambda+2,\alpha+\lambda-1)$.
The remaining statements follows from (1) and (2).\newline
(5). Remark \autoref{Rem:-onRVermalambdaalpha} and (1)-(3) imply $\RVerma(\lambda,<)^- \cong \RVerma(\lambda,>)$
and $\RVerma(\lambda,=)^-\cong \RVerma(\lambda,=)$. $\RVerma(\lambda,<)^{\vee} \cong \RVerma(\lambda,>)$
follows from the proof of (1) and the fact that the definition of $(\cdot)^{\vee}$ 
implies $\RVerma(\lambda,\alpha)_{\nu-2} f = 0$ if and only if $(\RVerma(\lambda,\alpha)^{\vee})_{\nu}e = 0$.
\end{proof}
The subquotient structure of the $\RVerma(\lambda,\alpha)$ is easily deduced from this lemma together
with the subquotient structure of the Verma and dual Verma module. 
E.g. for $\lambda \geq 2$ the simple modules $\MVerma(-\lambda), \MVerma(-\lambda)^-$ and $\Lsimple(\lambda-2)$
occur with multiplicity one, while for $\lambda = 1$
the simple modules $\MVerma(-1)$ and $\MVerma(-1)^-$ occur with multiplicity one. 

\subsubsection{$\Ho^0$ and $\Ho^1$ of $j_{\cdot}\Omega_{\C^{\times}}$}
We find 
\begin{Thm}\label{Thm:jdotOmega}
\begin{align*}
\Ho^0(j_{\cdot}\Omega_{\C^{\times}})=\Omega_{\C^{\times}}(\C^{\times})\cong 
\begin{cases} \RVerma(\lambda,=) & \lambda \leq 0 \\ \RVerma(\lambda,=)^{\vee} & \lambda \geq 2 \\ \MVerma(-1)\oplus\MVerma(-1)^-
& \lambda = 1\end{cases}\qquad \Ho^1(j_{\cdot}\Omega_{\C^{\times}})=0\;.
\end{align*}\qed
\end{Thm} 
Since $\Ho^1(j_{x\cdot}\Omega_{\C_x})=0$ the long exact sequence of cohomology of \eqref{eq:sesforjdotOmegaCtimes} 
and \eqref{eq:twistedjxextensions} give the exact sequence
\begin{align*}
0 \rightarrow \MVerma(\lambda-2)^{\vee} \rightarrow \Ho^0(j_{\cdot}\Omega_{\C^{\times}}) \rightarrow \MVerma(-\lambda)^- \rightarrow 0\;.
\end{align*}
Applying $(\cdot)^-$ we obtain the exact sequence
\begin{align*}
0 \rightarrow \MVerma(\lambda-2)^{-\vee} \rightarrow \Ho^0(j_{\cdot}\Omega_{\C^{\times}}) \rightarrow \MVerma(-\lambda) \rightarrow 0\;.
\end{align*}

\subsubsection{$\Ho^0$ and $\Ho^1$ of $j_!\Omega_{\C^{\times}}$}
We find
\begin{Thm}\label{Thm:j!Omegaalpha}
\begin{align*}
\Ho^0(j_!\Omega_{\C^{\times}})\cong \begin{cases} \RVerma(\lambda,=) & \lambda \geq 2\\ \MVerma(-\lambda)\oplus\MVerma(-\lambda)^- &
\lambda \leq 1 \end{cases}\quad
\Ho^1(j_!\Omega_{\C^{\times}})\cong \begin{cases} 0 & \lambda \geq 1 \\ \Lsimple(-\lambda) & \lambda \leq 0 \end{cases}\;.
\end{align*}\qed
\end{Thm}
Since $\Ho^1(\iota_{z*}\C)=0$ the
long exact sequence of cohomology of \eqref{eq:sesforjshriekOmegaCtimes} and \eqref{eq:twistedjxextensions} give
the exact sequences
\begin{align*}
& 0 \leftarrow \MVerma(\lambda-2) \leftarrow \RVerma(\lambda,=) \leftarrow \MVerma(-\lambda)^- \leftarrow 0 &\lambda \geq 2\\
& 0 \leftarrow \MVerma(-\lambda) \leftarrow \MVerma(-\lambda)\oplus\MVerma(-\lambda)^- \leftarrow \MVerma(-\lambda)^- \leftarrow 0
& \lambda \leq 1\;.
\end{align*}

\subsubsection{$\Ho^0$ and $\Ho^1$ of $j_{!x\cdot z}\Omega_{\C^{\times}}$
and $j_{\cdot x!z}\Omega_{\C^{\times}}$}
We find
\begin{Thm}\label{Thm:j!dottwisted}
\begin{align*}
& \Ho^0(j_{!x\cdot z}\Omega_{\C^{\times}}) \cong \RVerma(\lambda,<)\quad \Ho^1(j_{!x\cdot z}\Omega_{\C^{\times}})=0\\
& \Ho^0(j_{\cdot x !z}\Omega_{\C^{\times}})\cong \RVerma(\lambda,>) \quad \Ho^1(j_{\cdot x! z}\Omega_{\C^{\times}})=0\;.
\end{align*}\qed
\end{Thm}
The statement for $j_{!x\cdot z}$ implies the one for $j_{\cdot x!z}$ because of $j_{\cdot x!z}\Omega_{\C^{\times}} \cong (j_{!x\cdot z}\Omega_{\C^{\times}})^-$,
Lemma \autoref{Lemma:autoequivalenceintertwinedbyHo} and Lemma \autoref{Lemma:RVermalambdaalpha}(5).

\begin{Rem}
The embeddings $\kappa = \iota_{x,z}, j_{x,z},j: Y \hookrightarrow \Proj^1$ we consider are all affine.
Let $\sh{M}$ be a $\sh{D}_Y$-module. Then $\Ho^1(Y,\sh{M})=0$ implies in a uniform way that the direct image $\kappa_* \sh{M}
= \iota_{x,z*}\sh{M}, j_{x,z\cdot}\sh{M}, j_{\cdot}\sh{M}$ in $\mod \sh{D}(\lambda)$ satisfies $\Ho^1(\kappa_*\sh{M})=0$. 
\end{Rem}

\begin{Rem}\label{Rem:BBtwisted}
The assignment $\sh{M} \mapsto \sh{M}\otimes_{\sh{O}}\sh{O}(\lambda)$
defines an exact equivalence $\modulecat\sh{D} \rightarrow \modulecat\sh{D}(\lambda)$.
Thus, on the level of $\sh{D}$-modules, all twists $\lambda$ are equivalent. 
As mentioned in \autoref{ssec:leftandrightDmod} the twist $\lambda = 2$ corresponds to untwisted left $\sh{D}$-modules
because $\Omega \cong \sh{O}(-2)$. The Beilinson-Bernstein equivalence \cite{BB81} applies in the case $\lambda \geq 2$: Then $\Ho^1=0$ and
$\Ho^0: \modulecat\sh{D}(\lambda) \rightarrow \modulecat(\Uea\mathfrak{sl}_2)_{\chi_{\lambda}}$
is an exact equivalence. The vanishing of $\Ho^1$ in the above examples is consistent with this. 
The twist $\lambda = 1$ is usually referred to as singular, see \eqref{eq:twistedjxextensions} and Lemma \autoref{Lemma:RVermalambdaalpha}(3) for examples of how this twist differs from the others. 
One defines the dual $\Dual \sh{M}$ of a holonomic $\sh{D}(\lambda)$-module $\sh{M}$ via the 
above equivalence $\modulecat\sh{D} \rightarrow \modulecat\sh{D}(\lambda)$. 
Then $\Dual$ defines an exact contravariant auto-equivalence of 
the category of holonomic $\sh{D}(\lambda)$-modules. In the above examples $\Ho^0$ sends $\Dual$ to $(\cdot)^{\vee}$ in the
case $\lambda \geq 2$, but we are not aware of a reference proving that such a statement is true in general.
\end{Rem}

\section{Outlook} \label{sec:outlook}
\subsection{Generalization to a semisimple group}\label{ssec:generalizationtosemisimple}
Let us describe the setup of a possible generalization of this work
from $\SL_2$ to any semisimple algebraic group $G$ over $\C$. 
Let $B$ be a Borel subgroup of $G$, $T$ a maximal torus contained in $B$ and $U$ 
the unipotent radical of $B$. Let $\Phi^+$ be the set of positive roots of $G$ w.r.t. $B$.
For an element $w$ of the Weyl group $W = \Norm_G(T)/T$ we denote by $\dot{w}$
any preimage in the normalizer $\Norm_G(T)$ of $T$. All expressions
written below will be independent of such a choice unless mentioned otherwise. Consider the
flag variety $X=G/B$ of $G$. The stratification
of $X$ by $B$-orbits is the well-known Bruhat stratification.
It is given by $X = \bigsqcup_{w \in W}X_w$, 
where $X_w = B\dot{w}B$ is the Bruhat cell associated to $w$. Each $X_w$
is an affine space whose dimension is given by the length $\length(w)$ of $w$. 
\par
Fix a simple reflection $s_i \in W$ w.r.t. $B$. Then 
$B^{(i)} = B \cap \dot{s_i}B\dot{s_i}^{-1}$ defines a closed subgroup of $G$ 
that is contained in $B$ and contains $T$. Clearly
we have $B^{(i)} = T \ltimes U^{(i)}$, where $U^{(i)}$ is the subgroup of $U$
whose Lie algebra is spanned by the
root vectors associated to $\Phi^+\setminus\{\alpha_i\}$. 
Let $<$ be the Bruhat-Chevalley order on $W$.
One can show that the $B^{(i)}$-orbits in $X$ are given by $X_w$ and
$\dot{s_i} X_w$ for $s_i w \nless w$ and $X_w \cap \dot{s_i} X_w$ for $s_i w < w$. 
In particular, the stratification of $X$ by $B^{(i)}$-orbits refines the Bruhat stratification
and consists of finitely many locally closed subvarieties of $X$, which is in contrast
with the stratification of $X$ by $T$-orbits. 
Note that $\dot{s_i}$ defines an automorphism of $X$, depending on the
choice of $\dot{s_i}$, that interchanges
the $B^{(i)}$-orbits $X_w$ and $\dot{s_i}X_w$ and leaves the $B^{(i)}$-orbit
$X_w \cap \dot{s_i}X_w$ invariant. 
Observe that for $s_i w < w$ the $B$-orbit $X_w$ decomposes into two $B^{(i)}$-orbits,
namely $X_w = (X_w \cap \dot{s_i}X_w)\sqcup \dot{s_i} X_{s_i w}$. 
In fact $X_w \cap \dot{s_i} X_w$ is
isomorphic as a variety to $(\Aff^1\setminus\{0\})\times \Aff^{\length(w)-1}$.
Let us emphasize how we recover for $G=\SL_2$
the setting of the main text: In this case we have $X=\Proj^1$, our subgroup is the torus, $B^{(i)}=T$, and its
orbits are $\{0\}$, $\{\infty\}$ and $\C^{\times}$. For a suitable choice the automorphism $\dot{s_i}$
of $\Proj^1$ coincides with the involution $I$ defined in \autoref{ssec:autoequivalences-}.\par 

Coming back to the general $G$, given a parameter $\alpha \in \C$  it is clear how to
define a local system on the $B^{(i)}$-orbit $X_w \cap \dot{s_i} X_w$, generalizing $\Omega^{(\alpha)}_{\C^{\times}}$ 
from \autoref{ssec:defofOmegaalpha}, and the
different extensions to $X$ as $\lambda$-twisted $\sh{D}$-module. (The twist $\lambda$ is now an arbitrary integral weight 
of $G$ w.r.t. $T$.) It would be interesting
to identify, depending on the parameters $\alpha$ and $\lambda$, the sheaf cohomology groups as concrete modules over the Lie algebra $\Lie G$ of $G$ thereby generalizing the results of the main text. Here, the word concrete refers to having
some algebraic construction of the module, e.g. as a module induced from a proper subalgebra. 
Concerning the remaining $B^{(i)}$-orbits, it is known that from the $X_w$ one obtains (the usual highest weight) Verma and dual Verma $\Lie G$-modules,
while the $\dot{s_i}X_w$ are supposed to give $\Lie G$-modules differing from
the previous ones by a twist by the automorphism of $\Lie G$
defined by $\dot{s_i}$. 

\subsection{Case of the affine Kac-Moody Lie algebra $\widehat{\mathfrak{sl}_2}$}
Let us finally comment on similar constructions for the affine Kac-Moody Lie algebra $\widehat{\mathfrak{sl}_2}$. 
The discussion is parallel
to and depends on \autoref{ssec:generalizationtosemisimple}. Recall
that $\widehat{\mathfrak{sl}_2}$ is constructed from a triple $\left(\mathfrak{h},(\alpha_i)_{i \in \{0,1\}},
(\alpha_i^{\vee})_{i \in \{0,1\}}\right)$, where $\mathfrak{h}$ is a three dimensional
$\C$-vector space and the generalized Cartan matrix $\begin{pmatrix} 2 & -2\\ -2 & 2\end{pmatrix}$.
It has a triangular decomposition $\widehat{\mathfrak{sl}_2} = \mathfrak{n}^-\oplus \mathfrak{h}\oplus \mathfrak{n}$
and elements $f_i \in \mathfrak{n}^-$, $e_i \in \mathfrak{n}$, $i \in \{0,1\}$.  The flag variety $X$ of $\widehat{\mathfrak{sl}_2}$
has well-known formulations either as an ind-projective ind-variety \cite{BD91} or as a scheme \cite{Kas90} over $\C$. 
In the following we work with the latter version. The Weyl group $W$ of $\widehat{\mathfrak{sl}_2}$
is generated by $s_0$, $s_1$ with relations $s_0^2=s_1^2=1$. 
In \cite{KT95} it is explained how to attach to $w \in W$ a (finite dimensional)
Bruhat cell $X_w \cong \Aff^{\length(w)}$ that is locally closed in $X$.  
One has a replacement for $\dot{s_1}$
and for $s_1 w < w$ the intersection $X_w \cap \dot{s_1} X_w$ is still defined and isomorphic
to $(\Aff^1\setminus\{0\})\times \Aff^{\length(w)-1}$. Due to \cite{KT95}
one disposes of a category of $\lambda$-twisted right $\sh{D}$-modules on $X$ with support in the closure $\overline{X_w}$
of $X_w$ in $X$. 
To the objects of this category \cite{KT95} attach cohomology groups and construct a $\widehat{\mathfrak{sl}_2}$-action on these. Consequently it makes sense
to try to identify the cohomology groups of the $\sh{D}$-module extensions of the
local system on $X_w \cap \dot{s_1} X_w$ analogous to $\Omega^{(\alpha)}_{\C^{\times}}$ as
concrete $\widehat{\mathfrak{sl}_2}$-modules, completely parallel to what was
proposed in \autoref{ssec:generalizationtosemisimple}.  In fact, the 
corresponding identification of the cohomology groups for $X_w$ instead of $X_w \cap \dot{s_1} X_w$ and antidominant $\lambda$ constitutes the main theorem
\cite{KT95}[Theorem 3.4.1] of \cite{KT95}.  
Let us consider the simplest case $w=s_1$, then
$X_w \cap \dot{s_1} X_w \cong \Aff^1\setminus\{0\}$. We can then prove that for $-\lambda(\alpha_1^{\vee})\geq 2$
the global sections of the $!$- resp. $*$-extension form an induced module of the form $\Uea \widehat{\mathfrak{sl}_2} \otimes_{\Uea \mathfrak{p}_1} (\C_{\lambda}\otimes_{\C} M)$ with $M=\RVerma(-\lambda(\alpha_1^{\vee}),\alpha)$ as defined in \eqref{eq:RVermaalphalambda} resp. $M=\RVerma(-\lambda(\alpha_1^{\vee}),\alpha)^{\vee}$. Here the subalgebra $\mathfrak{p}_1=\mathfrak{n}\oplus \mathfrak{h}\oplus \C f_1$ acts on $\C_{\lambda}\otimes_{\C} M$ via the map of Lie algebras 
\begin{align*}
\mathfrak{p}_1 \twoheadrightarrow \{h \in \mathfrak{h}\; \vert\; \alpha_1(h)=0\}\oplus (\C f_1 \oplus
\C \alpha_1^{\vee} \oplus \C e_1)\;. 
\end{align*}
Note that this induced module is neither a highest nor a lowest weight module. 
Such $\widehat{\mathfrak{sl}_2}$-modules were introduced and analyzed in \cite{SS97}
and \cite{FST98} under the name relaxed Verma modules. 
It is natural to ask whether a similar result holds for general $w$ and we hope
to address this question in a future publication.

\bibliographystyle{alpha}
\bibliography{references}

\end{document}